\documentclass[12pt]{amsart}
\usepackage{amsmath,amscd,amsthm,amsfonts, amssymb,amsxtra, mathrsfs,enumerate,relsize}
\usepackage{calrsfs}
\usepackage[us,12hr]{datetime}
\usepackage[all]{xy} \SelectTips{cm}{}
\usepackage{hyperref}
 \hypersetup{colorlinks=true,citecolor=blue}
   \usepackage{graphicx}
     \usepackage[margin=1.5in]{geometry}
   \usepackage[font=scriptsize]{caption}
\usepackage{todonotes}
\usepackage{tikz}
\usetikzlibrary{positioning}
\CDat
\usepackage{scalerel}[2016/12/29]

\dedicatory{In memory of Friedhelm Waldhausen.}
\newtheorem{thm}{Theorem}[section]  
\newtheorem*{un-no-thm}{Theorem}
\newtheorem{cor}[thm]{Corollary}     
\newtheorem{lem}[thm]{Lemma}         
\newtheorem{prop}[thm]{Proposition}

\newtheorem{bigthm}{Theorem}

\theoremstyle{definition}
\newtheorem{defn}[thm]{Definition}   

\theoremstyle{definition}

\theoremstyle{definition}

\theoremstyle{remark}
\newtheorem{rem}[thm]{Remark}        

\newtheorem*{acks}{Acknowledgements}
\newtheorem*{out}{Outline}
\newtheorem{rems}[thm]{Remarks}
 
\newtheorem{ex}[thm]{Example}

\DeclareMathOperator{\Top}{Top}
\DeclareMathOperator{\Sp}{Spectra}

\DeclareMathOperator{\tr}{tr}

\begin{document}
\title{Non-simply connected Poincaré surgery}
\date{\today}
\author{John R.\ Klein}
\address{Wayne State University, Detroit, MI 48202}
\email{klein@math.wayne.edu}
\begin{abstract}We give homotopy-theoretic proofs of the principal
results of non-simply connected Poincar\'e surgery.
These include the Fundamental Theorem of Poincaré surgery, the Poincaré transversality exact sequence, and the construction of Poincaré handle structures. All results apply in dimensions~$\ge 8$. \end{abstract}
\maketitle
\setlength{\parindent}{15pt}
\setlength{\parskip}{1pt plus 0pt minus 1pt}

\def\Sp{\text{\bf Sp}}
\def\vo{\varOmega}
\def\vs{\varSigma}
\def\smsh{\wedge}
\def\flush{\flushpar}
\def\id{\text{id}}
\def\dbslash{/\!\! /}
\def\codim{\text{\rm codim\,}}
\def\:{\colon}
\def\holim{\text{holim\,}}
\def\hocolim{\text{hocolim\,}}
\def\Bbb{\mathbb}
\def\bold{\mathbf}
\def\Aut{\text{\rm Aut}}
\def\cal{\mathcal}
\def\sec{\text{\rm sec}}
\def\gda{G\text{\rm -}\delta\text{\rm -}\alpha}
\def\PDD{\text{\rm pd\,}}
\def\PD{\text{\rm P}}
\def\stableto {\,\, \mapstochar \!\!\to}

\setcounter{tocdepth}{1}
\tableofcontents


\section{Introduction}
This paper continues the work begun in \cite{klein_surgery1}. In that earlier paper, we established several of the principal results of Poincaré surgery in the simply connected case. Here, we extend those results to the non-simply connected setting. The methods developed in this work exemplify Browder’s maxim that “a homotopy theoretical problem should have a homotopy theoretical solution” \cite{Quinn_mimeo}. 
The current approach differs considerably from \cite{klein_surgery1} (cf.~
\S\ref{subsec:approach} below).
For additional historical context, we refer the reader to \cite[intro.]{HV} and \cite[\S1.3]{klein_surgery1}.

The results of this paper, together with those of
\cite{klein_surgery1}, complete a homotopy-theoretic program
for Poincar\'e surgery that goes back more than fifty years.
In 1972, Quinn announced a homotopical approach to Poincar\'e
surgery \cite{Quinn}, but a crucial desuspension argument was subsequently
found to contain gaps, and the relevant statements are in fact
false.  Later attempts at the fundamental theorem made substantial use of manifold theory \cite{Jones}, \cite{Hodgson_surgery}, \cite{HV}.  Thus the
question remained whether the principal results of Poincar\'e
surgery could be proved intrinsically in the category of
Poincar\'e spaces.  The simply connected case was treated in
\cite{klein_surgery1}.  The present paper supplies the
non-simply connected case, and thereby gives a manifold-free
proof of the fundamental theorem for arbitrary fundamental
groups, together with corresponding results on Poincar\'e
transversality and Poincar\'e handle structures.

\subsection{The Fundamental Theorem} Let $X$ be a connected Poincar\'e duality space of dimension $d$. 
Then $X$ has a stable spherical fibration $\xi$, the
Spivak normal fibration, which is characterized up to fiber homotopy equivalence by the condition that there exists
a stable homotopy class $\alpha\in \pi_d(X^\xi)$, called the {\it normal invariant}, whose Hurewicz image $[X] \in H_d(X;L)$
is a fundamental class for $X$. Here
$X^\xi$ denotes the Thom spectrum of $\xi$ and $L$ is the orientation sheaf of $X$. The pair $(\xi,\alpha)$ is unique up to fiber homotopy equivalence
and determines the pair $(L,[X])$.

A Poincar\'e {\it normal map} 
\[
(f,b)\: M \to X
\]
consists of a degree one map of Poincar\'e duality  spaces $f\: M\to X$ of dimension $d$ together with a stable fiber homotopy equivalence
 $b\: \xi_M \simeq f^\ast \xi$, where $\xi_M$ is the 
Spivak normal fibration of $M$. 
Let 
\[
\cal N(X,\xi) 
\]
be the set of bordism classes of Poincar\'e normal maps with target 
$(X,\xi)$.
The surgery obstruction in this setting is well-defined and is a function
\[
\sigma\colon \cal N(X,\xi) \longrightarrow L_d(\pi,w),
\]
where $\pi=\pi_1(X)$, $w=w_1(\xi)$ is the orientation character, and
\[
L_d(\pi,w):=L_d^h(\pi,w)
\]
is Wall's $L$-group with the $h$-decoration \cite{Ranicki_exact}, \cite{Ranicki_ags}. We suppress the decoration
throughout.
If an element of $\cal N(X,\xi)$ admits a representative
$(f,b)$ in which $f\: M\to X$ is a homotopy equivalence,  then $\sigma(f,b) = 0$. The ``Fundamental Theorem'' of Poincar\'e surgery
asserts that the converse is true when $d \ge 8$:
 
 \begin{bigthm}[Fundamental Theorem] \label{bigthm:fund-thm} 
Assume $d\ge 8$. If $\sigma(f,b) = 0$, then $(f,b)$ is normally cobordant to a  homotopy equivalence.
  \end{bigthm}

  \begin{rem} Compare \cite[cor.~1.4]{Quinn}, \cite[thm.~3.4, rem.~3.17, cor.~6.2]{Jones}, \cite[thm.~4.1]{Hodgson_surgery}, and
 \cite[thm.~5.1]{HV}. With the exception of the Bulletin announcement \cite{Quinn},\footnote{The announcement \cite{Quinn}, which seems to have
 been based on the unpublished manuscript \cite{Quinn_mimeo}, was later shown to contain an irreparable gap.} the other published works make essential use
 of smooth manifolds. The current paper provides a homotopy-theoretic approach.
  If $\pi$ is trivial, then $d \ge 7$ suffices by \cite[thm.~A]{klein_surgery1}.
  \end{rem}


\begin{bigthm}[Absolute Realization] \label{bigthm:sigma-onto} If $ d \ge 8$, then $\sigma$ is  onto.
\end{bigthm}

\begin{rem} When $d$ is even, this was already shown to be the case in \cite[thm.~B]{klein_surgery1}.
\end{rem}

\subsection{Poincar\'e transversality}
Let $X$ be a connected space of finite type equipped with a stable spherical fibration $\xi$. Then 
the Poincar\'e bordism group $\Omega_d^{P}(X,\xi)$ is defined. Its generators
are maps $f\:P\to X$ such that $P$ is a Poincar\'e duality space and $f^\ast\xi$ is identified with the Spivak normal fibration.

Let  
\[
c\: \Omega_d^{P}(X,\xi) @>>> \pi_{d}(X^\xi)
\] denote the Poincar\'e duality space version of the
 Pontryagin-Thom homomorphism:
It assigns to $f\: P \to X$ the composition
\[
S^d @> \alpha >> P^{f^\ast \xi} @> f^\xi >> X^\xi \, ,
\]
where $\alpha$ is the stable normal invariant of the Spivak fibration and $f^\xi$ is the 
map of Thom spectra induced by $f$. In  contrast with the manifold case, 
the homomorphism $c$
need not be an isomorphism, reflecting the lack of a Thom transversality result in the Poincar\'e category.

In \cite{klein_surgery1} we constructed an exact sequence of abelian groups
\[
\cdots \to \cal F_d(X,\xi)  \to \Omega_d^{P}(X,\xi) @> c >> \pi_{d}(X^\xi) \to  \cal F_{d-1}(X,\xi) \to \cdots \, ,
\]
and a natural homomorphism 
\[
\sigma'\: \cal F_d(X,\xi) \to L_d(\pi,w)\, ,
\] where $\pi = \pi_1(X)$ and $w\: \pi \to \{\pm 1\}$ is the orientation character
of $\xi$. We proved that $\sigma'$ is an isomorphism when $\pi$ is trivial and $d \ge 7$.
In the case of non-trivial $\pi$ we have the following:

\begin{bigthm}[Poincar\'e Transversality Obstruction]
\label{bigthm:stronger-exactness} If  $d\ge 8$, then the  homomorphism $\sigma'$ is an isomorphism.
\end{bigthm}

\subsection{Handlebodies} Let $M$ be a Poincar\'e duality space of dimension $d$. 
A {\it handle structure} on $M$ consists of a homotopy equivalence $M \simeq M_d$ and a filtration
\[
\emptyset = M_{-1} \subset M_0 \subset \cdots \subset M_d
\]
in which $M_j$ is obtained from $M_{j-1}$ by attaching a finite number of $j$-handles
\[
(H,\partial_0 H) = (D^j \times D^{d-j},S^{j-1} \times D^{d-j})
\]  
to $\partial M_{j-1}$. More precisely, if $r$ is the number of $j$-handles, then $M_j$ is the result of
attaching $\amalg_{i=1}^r H$ to $M_{j-1}$  along 
a Poincar\'e embedding $\amalg_{i=1}^r \partial_0 H  \to \partial M_{j-1}$.

\begin{bigthm}[Poincar\'e Handlebodies] \label{bigthm:handlebody} Every Poincar\'e duality  space of dimension $d \ge 8$
admits a handle structure.
\end{bigthm}

\begin{rem} A handle structure is a special case of a {\it patch structure} in the sense of \cite{Jones}.
\end{rem}

\subsection{Remarks on the approach}  \label{subsec:approach}
With the exception of the existence of handlebody structures, the main results of this paper parallel those of \cite{klein_surgery1}. The homotopy-theoretic tools are largely similar, though their implementation here requires greater delicacy. On the homotopy theory side, we draw upon the Poincar\'e embedding results of \cite{Klein_haef}, \cite{Klein_haef2}, \cite{Klein_compression}, \cite{Klein_compression2} as well as the homotopical intersection theory developed in \cite{Klein-Williams_Intersection1}. Because we work in the presence of a fundamental group $\pi$, the naive stable homotopy theory of spaces with a $\pi$-action becomes an essential ingredient. In \cite{klein_surgery1}, the simply connected ``fundamental theorem'' was established by an argument paralleling that of Browder \cite{Browder}. In the present, non-simply connected setting, we instead follow Wall's approach \cite{Wall}, incorporating certain simplifications due to Ranicki \cite{Ranicki_ags}.

When the dimension is even, $d = 2k$, the situation is comparatively straightforward. We employ a homotopical version of the double point invariant 
$\mu$, defined on the middle-dimensional surgery kernel (cf.~ \cite{Klein_sphere} in the 1-connected case). In Wall's framework, the function $\mu$ 
serves as the obstruction to deforming a smooth immersion of a framed sphere into a smooth embedding via regular homotopy. For the intersection 
pairing $\lambda$, we use the methods of \cite{Klein-Williams_Intersection1}. The outcome is a kernel form $(K_k(M), \lambda, \mu)$, whose 
cobordism class represents the surgery obstruction in the category of Poincar\'e normal maps.

The odd-dimensional case $d = 2k + 1$ turns out to be  substantially more intricate. I was unable to directly adapt Wall's method to the Poincar\'e duality context, but a key insight emerged from studying \cite[ch.~12]{Ranicki_ags}. The central idea is to construct the kernel formation of a Poincar\'e normal map via a Heegaard decomposition. Achieving this requires translating Ranicki's construction into the Poincar\'e duality space setting. Such a  translation is  carried out in Appendix~\ref{sec:Heegaard-appendix}.

\begin{out} Section \ref{sec:prelim} provides some of the  background for reading the paper. Section \ref{sec:embeddings} recalls the definition
of codimension zero Poincar\'e embeddings and restates some of the known results in a relative setting. Section \ref{sec:normal-maps} reviews
the four variant notions of a Poincar\'e normal map that appear in \cite{klein_surgery1}. Section \ref{sec:forms} recalls the definition of forms and formations over $\Bbb Z[\pi]$; the material appearing here is recalled from 
\cite{Ranicki_ags}. Section \ref{sec:Hopf}
reviews the equivariant stable Hopf invariant, following \cite{klein2026stablehopf}. Section
 \ref{sec:kernels} introduces the kernel modules associated with a highly connected Poincar\'e normal map. Section \ref{sec:kernel-form}
 constructs the kernel form of a highly connected Poincar\'e normal map when $d= 2k$ is even; the construction relies on the material of section \ref{sec:Hopf}.
Section \ref{sec:Heegard-normal} provides the definition of the kernel formation of a Poincar\'e normal map when $d=2k+1$ is odd. The definition relies
on the existence of a Heegaard splitting, whose existence is established in Appendix \ref{sec:Heegaard-appendix}. Section \ref{sec:surgery-obstruction}
defines the surgery obstruction of a Poincar\'e normal map. Section \ref{sec:realization} proves a Wall realization result in odd dimensions
in the Poincar\'e category and uses it to deduce Theorem \ref{bigthm:sigma-onto}. Sections \ref{sec:d=2k} and \ref{sec:d=2k+1} contain the proof of Theorem 
\ref{bigthm:fund-thm}. Section \ref{sec:transversality} concerns the proof of Theorem \ref{bigthm:stronger-exactness}. Section \ref{sec:handlebody}
contains the proof of Theorem \ref{bigthm:handlebody}.

We include two appendices. Appendix \ref{sec:top-cell} constructs a top cell splitting of a Poincar\'e normal map. Appendix \ref{sec:Heegaard-appendix}
provides the construction of a Heegaard decomposition of a Poincar\'e normal map.
\end{out}

\begin{acks}
This paper is dedicated to the memory of Friedhelm Waldhausen,
who passed away on November 21, 2024. During the years 1992--1998,
I had the privilege of serving as his assistant at Bielefeld
University. Although neither his research nor mine was focused on
surgery theory during that period, many of the ideas and methods
used in this work can ultimately be traced to what I learned from
him. His death marked the loss not only of a remarkable
mathematician, but also of a gifted poet.

I am deeply indebted to the work and influence of Andrew Ranicki,
who passed away on February 21, 2018. My understanding of the
algebraic topology of manifolds was shaped to a considerable
extent by conversations with him during visits to Edinburgh and
Oberwolfach. I should also mention that the difficulties I had
encountered in understanding the passage from geometry to algebra in \cite[ch.~6]{Wall} were finally
resolved after I discovered Ranicki's book
\cite{Ranicki_ags}.
\end{acks}

\section{Preliminaries} \label{sec:prelim}
\subsection{Spaces}
Let $\Top$ denote the category of compactly generated weak Hausdorff spaces, equipped with
the Quillen model structure in which a map $X\to Y$ is a weak equivalence (resp.~fibration)
if it is a weak homotopy equivalence (resp.~Serre fibration) \cite[thm.~2.4.25]{Hovey}. 
A map $X\to Y$ is a cofibration if it is the result of attaching cells to $X$, or if it is the retract of such a map.
We will implicitly work with cofibrant spaces and apply cofibrant approximation whenever necessary to guarantee
cofibrancy. Similarly, let $\Top_\ast$ denote the model category of based spaces.

A non-empty space $X$ 
is {\it $0$-connected} if it is path connected.
It is $r$-connected with  $r > 0$, if it is $0$-connected and the set of homotopy classes
$[S^k,X]$ has one element for $k\le r$.
A map of non-empty (possibly unbased) spaces is
{\it $r$-connected} if its homotopy fiber, taken with respect to
all choices of basepoint, is $(r{-}1)$-connected. Let $I = [0,1]$ denote the unit interval.

A commutative square of spaces
\[
\xymatrix{
A \ar[r] \ar[d] & C \ar[d] \\
B \ar[r] & D
}
\] 
is a {\it homotopy pushout} if the induced map
\[
B\times \{0\} \cup A \times I \cup C\times \{1\} \to D
\]
is a weak equivalence, where the source is the double mapping cylinder of the diagram $B \leftarrow A \to C$.

A {\it triad} of spaces is an ordered triple 
\[
(X;X_0,X_1)
\] in which $X_i \to X$ are cofibrations.
We sometimes let $X_{01} := X_0 \times_X X_1$ denote the intersection of $X_0$ with $X_1$ in $X$.

A space $X$ is {\it homotopy finite} if it has the homotopy type of a finite CW complex.
It is of {\it finite type} if it has the homotopy type of a CW complex having finitely many cells in each dimension.
A {\it $k$-skeleton} for a space $X$  is a $k$-connected map of spaces $Y\to X$ in which $Y$ is a CW complex of dimension $\le k$.

A {\it stable map} $X \to  Y$ of based spaces is a map of suspension spectra $\Sigma^\infty X\to \Sigma^\infty Y$.
The abelian group of homotopy classes of stable maps $X\to Y$ is denoted by $\{X,Y\}$.
When $X$ is a finite complex,  a stable map is represented by a map of based spaces
 $\Sigma^j X \to \Sigma^j Y$ for some $j \ge 0$, where $\Sigma^j X$ is the reduced $j$-fold suspension of $X$. 
We identify two such maps if they agree after some iterated suspension.

\subsection{Poincar\'e duality} 
A  space $P$ is a {\it Poincar\'e duality space} 
of  dimension $d$ if there exists a pair
\[
(L,[P])
\]
in which $L$ is a  bundle  of coefficients that is locally isomorphic
to $\Bbb Z$, and  $[P] \in H_d(P;L )$ is a homology class
such that the associated cap product homomorphism
\[
\cap [P]\:H^*(P;\cal B) \to H_{d{-}*}(P;\cal B \otimes L )
\]
is an isomorphism in all degrees for every local coefficient bundle $\cal B$ \cite{Wall_PD}.
If the pair $(L,[P])$ exists, then it
is defined up to unique isomorphism; $L$ is called
the {\it orientation sheaf} and $[P]$ the {\it fundamental class}.

Similarly, given a Poincar\'e space $\partial P$
of dimension $d-1$, one has the notion of {\it Poincar\'e pair} $(P,\partial P)$ of dimension $d$,
where now $[P] \in H_d(P,\partial P;L )$ induces an isomorphism
\[
\cap [P]\:H^*(P;\cal B) \to H_{d{-}*}(P,\partial P; \cal B\otimes L)\, 
\]
(cf.~\cite[thm.~B]{Klein-Qin-Su}).
In this paper, we will assume that Poincar\'e spaces (pairs)
are homotopy finite, i.e., they have the homotopy type of a finite CW complex
(resp.~finite CW pair).

\subsection{Poincar\'e triads} 
Let $X$ be a Poincar\'e duality space of dimension $d$. A {\it codimension one splitting} of $X$ is a triad
\[
(X;X_0,X_1)
\]
such that $(X_0,X_{01})$ and $(X_1,X_{01})$ are $d$-dimensional Poincar\'e pairs and
\[
X = X_0 \cup_{X_{01}} X_1\, .
\]
A {\it Poincar\'e triad} is a Poincar\'e pair $(P,\partial P)$ such that $\partial P$ is equipped with a codimension one splitting
\[
(\partial P;\partial_0 P,\partial_1 P)\, .
\]
We denote these data by the triad
\[
(P;\partial_0 P,\partial_1 P)\, .
\]
\begin{ex}[Top cell decompositions] Let $X$ be a connected Poincar\'e duality space of dimension $d$. A {\it top cell decomposition} of
$X$ consists of a  choice of homotopy equivalence
\[
X \simeq X_0 \cup_f D^{d}\, ,
\]
where $f\:S^{d-1}\to X_0$ is a cofibration and $(X_0,S^{d-1})$ is a Poincar\'e pair of dimension $d$. For $d \ge 4$, $X_0$ may be chosen to have the homotopy type of a CW complex of dimension at most $d-1$.
Consequently, a top cell decomposition of $X$ yields the codimension one splitting
\[
(Y;X_0,D^d) \, ,
\]
where $Y = X_0 \cup_f D^d$.
Every connected Poincar\'e space of dimension $d \ge 3$ admits a top cell decomposition. Moreover, the decomposition
is unique up to weak equivalence of triads \cite[thm.~2.4]{Wall_PD}.
\end{ex}

\subsection{Connected sums}
Let $(P^d;P_0,D^d)$ and $(Q^d;Q_0,D^d)$ be codimension one splittings. 
The  {\it connected sum} is the $d$-dimensional Poincar\'e duality space
\[
P \sharp Q := P_0 \cup_{S^{d-1}} Q_0 \, .
\]
 More generally, if $P^d$ and $Q^d$ are
Poincar\'e duality  spaces, we may choose top cell decompositions $P \simeq  P_0 \cup_f D^d$ and $Q \simeq  Q_0 \cup_g D^d$ and define
$P\sharp Q$ using the associated codimension one splittings. Alternatively, we may define $P \sharp Q$
as the double mapping cylinder
\[
P_0 \times\{0\} \cup_{f} (S^{d-1} \times I) \cup_g Q_0 \times \{1\}\, .
\]
The connected sum is well-defined up to homotopy.

If $P$ and $Q$ aren't connected, then to define a connected sum we must equip the sets  $\pi_0(P)$ and $\pi_0(Q)$ with basepoints, 
and we may take the connected sum at the preferred basepoint components and the disjoint union 
with the remaining components to define $P\sharp Q$.

If $(P,\partial P)$ and $(Q,\partial Q)$ are $d$-dimensional Poincar\'e pairs equipped with codimension one splittings
$(\partial P; \partial_0 P,D^{d-1})$
and $(\partial Q;\partial_0 Q, D^{d-1})$, then  the {\it boundary connected sum} is the Poincar\'e pair
\[
(P \natural Q,\partial (P \natural Q)) =  (P \cup_{D^{d-1}} Q,\partial P \sharp \partial Q)\, .
\]

 \subsection{Stable spherical fibrations}
Let $BG$ be the classifying space of stable spherical fibrations \cite{Stasheff_BG}. 
This is the classifying space of the grouplike topological monoid $G$ consisting of
stable self equivalences of the zero sphere $S^0$.

By a {\it stable spherical fibration} over a space $X$, we mean a map
$\xi\: X\to BG$. An orientation sheaf $L$ on $X$ associated with $\xi$ is defined by 
the first Stiefel-Whitney class $w = w_1(\xi) \in H^1(X,\Bbb Z_2)$; the 
latter is the pullback of the generator of  $H^1(BG;\Bbb Z_2) \cong \Bbb Z_2$.
Note that  $w$ may be regarded as a homomorphism $\pi\to \{\pm 1\}$
called the {\it orientation character}, where $\pi$ is the fundamental group of $X$.

To relate the above notion of stable spherical fibration to the classical one, 
let $p\: EG\to BG$ be a universal classifying fibration in which $EG$ is a free contractible
$G$-space. Form the fiber product
 $\bar X = X \times_{BG} EG$; the latter is equipped with a free action of $G$. Let $S$ be the sphere spectrum. Then 
 $G$ acts on $S$ 
 and we obtain a parametrized spectrum
  \[
 \bar X\times_G S \to X
 \]
 whose fibers are spherical. The {\it Thom spectrum} of $\xi$ is
 \[
X^\xi := (\bar X_+) \smsh_{G} S\, .
 \]
Similarly, if $\tilde X \to X$ is the universal cover then we have a stable spherical fibration $\tilde X\to X \to BG$
and an associated  Thom spectrum with $\pi$-action $\tilde X^\xi$.
 
Then the {\it Thom class} of $\xi$
 \[
 U \in H^0(X^\xi ; L) := H^0_{\pi}(\tilde X^\xi; L)
 \]
induces the Thom isomorphism
 \[
 \cup U\: H^\ast(X;L) @>  \cong >> H^\ast(X^\xi;\Bbb Z ) \quad \text{ and } \quad
 \cap U\: H_\ast(X^\xi;\Bbb Z) @>  \cong >> H_\ast(X;L) \, .
\]

\subsection{The Spivak normal fibration} Let $X$ be a Poincar\'e duality space of dimension $d$. 
Let $\xi\: X \to BG$  be a  stable spherical fibration.
Then $\xi$ is  a {\it Spivak fibration} for $X$ if there exists a homotopy class $\alpha\in \pi_d(X^\xi)$
such that
\[
[X] := \alpha_\ast([S^d])\in H_{d}(X^\xi) \cong H_d(X; L)
\]
is a  fundamental class, where $[S^d] \in H_d(S^d;\Bbb Z)$ is the generator \cite{Spivak}, \cite{Browder_pd-surgery}, \cite{Klein_dualizing}. 
The homotopy class $\alpha$ is called the {\it normal invariant}.
The triple
\[
(X,\xi,\alpha)
\]
is called a {\it Spivak normal structure} \cite[def.~9.30]{Ranicki_ags}. It is unique up to fiber homotopy equivalence \cite[I.4.19]{Browder}, \cite{Klein_dualizing}.

If $(X, \partial X)$ is a Poincar\'e pair of dimension $d$, then one also has a Spivak normal structure
which can be defined by pulling back the Spivak fibration of the double $X \cup_{\partial X} X$
along the left inclusion $X \subset X\cup_{\partial X} X$. 
In this instance, the normal invariant $\alpha$ is a  homotopy class of pairs
$(D^d,S^{d-1}) \to  (X^\xi,(\partial X)^\xi)$ or equivalently, a homotopy class
\[
S^d \to X^\xi/(\partial X)^\xi\, .
\]

\section{Codimension zero Poincar\'e embeddings} \label{sec:embeddings}
Let $(P;\partial_0 P,\partial_1 P)$ and $(M;\partial_0 M,\partial_1 M)$ be Poincar\'e triads of dimension $d$
such that $(\partial_0 M,\partial_{01} M) = (\partial_0 P,\partial_{01} P)$.

\begin{defn}[{\cite[\S2]{GK}}]
A (codimension zero)  {\it Poincar\'e embedding} of $P$ in $M$ (relative to $\partial_0 P$) consists of a $d$-dimensional Poincar\'e triad
\[
(C; \partial_1 P, \partial_1 M)
\]
and a weak equivalence 
\[
h\: P\cup_{\partial_1 P} C @> \simeq >> M
\]
which restricts to the identity on $\partial M = \partial_0 P \cup \partial_1 M$. 

 The Poincar\'e triad $(C;\partial_1 P,\partial_1 M)$  is called the  {\it complement} of the Poincar\'e embedding.
 
 The composition
\[
f\: (P,\partial_0 P) @>\subset >> (P\cup_{\partial_1 P} C,\partial_0 P) @> h>\simeq > (M,\partial_0 P)
\]
is called the {\it underlying map} of the Poincar\'e embedding. 

Informally, we declare that a map $f\: P \to M$ (extending the identity
on $\partial_0 P$) {\it Poincar\'e embeds} if there is a Poincar\'e embedding of $P$ in $M$ (relative to $\partial_0 P$)
with underlying map $f$.
\end{defn}

\begin{rem} With respect to the identification $M \simeq P\cup_{\partial_1 P} C$, we obtain a codimension one splitting
of the Poincar\'e pair $(M,\partial M)$:
\[
((M,\partial M);(P,\partial_0 P),(C,\partial_1 M)) \, .
\]
\end{rem}

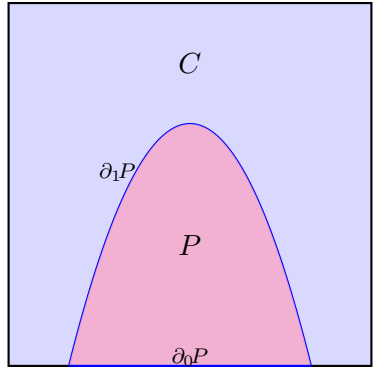
\begin{figure}
\centering
\begin{tikzpicture}[scale=.8]

\fill[blue!15] (-3,0) rectangle (3,6);
\draw[thick, black] (-3,0) rectangle (3,6);

\draw[thick, blue, domain=-2:2, samples=100] 
  plot (\x, {-(\x)^2 + 4});


\begin{scope}
\clip (-3,0) rectangle (3,6);
\fill[magenta!40, opacity=0.9] 
  plot[domain=2:-2, samples=100] (\x, {-(\x)^2 + 4}) 
  -- cycle;
\end{scope}

\draw[thick, blue] (-2,0) -- (2,0);

\node[font=\fontsize{8pt}{8pt}\selectfont] at (-1.2,3.2) {$\partial_1 \! P$};
\node[font=\fontsize{8pt}{8pt}\selectfont] at (0,.15) {$\partial_0 \! P$};
\node[font=\small] at (0,5) {$C$};
\node[font=\small] at (0,2) {$P$};
\end{tikzpicture}
\caption{A depiction of a Poincar\'e embedding of $P$ in $M$ relative to $\partial_0 P$, where $\partial_1 M$ is the 
 unlabelled part of the boundary.} 
\label{fig:embedding}
\end{figure}

\begin{rem} When $\partial_0 P = \emptyset$, the definition reduces to the one of \cite[\S2.4]{Klein_compression}.
In this instance we call it an {\it interior Poincar\'e embedding}.
\end{rem}

\begin{rem} Let $\iota\: \partial_1 P \cup \partial_1 M \to M$ be the inclusion. 
Then a Poincar\'e embedding defines an object of the category
$T_\iota$
whose objects are factorizations of  $\iota$ of the form $\partial_1 P \cup \partial_1 M \to C \to M$ and whose morphisms
are weak equivalences $C \to C'$ that are compatible with the factorizations. Let $E(P,M) \subset T_\iota$ be the full subcategory
consisting of objects which are Poincar\'e embeddings. Two Poincar\'e embeddings of $P$ in $M$ relative to $\partial_0 P$
are said to be {\it concordant} if they lie in the same path component of the realization $|E(P,M)|$.
\end{rem}

\subsection{Poincar\'e immersions} Let $P$ and $M$ be as above.
Let $j \ge 0$ be an integer. Then one has a Poincar\'e triad of dimension $d+j$:
\[
(P\times D^j; (\partial_0 P) \times D^j, S_P^j\partial_1 P)\, , 
\]
where $S_P^j\partial_1 P = ((\partial_1 P) \times D^j) \cup_{(\partial_1 P) \times S^{j-1}} P \times S^{j-1}$
is the $j$-fold unreduced fiberwise suspension of $\partial_1 P$ over $P$. Here we have implemented the Poincar\'e space decomposition
\begin{align*}
\partial (P\times D^j) &= ((\partial P) \times D^j)\cup (P \times S^{j-1})\\
&= ((\partial_0 P) \times D^j) \cup ((\partial_1 P) \times D^j) \cup (P \times S^{j-1})\\
&= ((\partial_0 P) \times D^j) \cup S_P^j\partial_1 P\, .
\end{align*}

\begin{defn}[{cf.~ \cite[\S1.4]{Klein_compression2}}] A {\it Poincar\'e immersion} of $P$ in $M$ is a Poincar\'e embedding
\[
P \times D^j \to M \times D^j
\]
(relative to $(\partial_0 P) \times D^j$) for some $j \ge 0$.

In particular, a Poincar\'e embedding is also a Poincar\'e immersion. One says that a Poincar\'e immersion of $P$ in $M$  {\it compresses}
 if, up to concordance, it arises from a Poincar\'e embedding of $P$ in $M$.
\end{defn}

Let $\nu_P$ denote the Spivak normal fibration of $(P;\partial P)$. 
The following result is  Theorem B of \cite{Klein_compression2}. 

\begin{thm} \label{thm:Smale-Hirsch}  Let $f\: (P,\partial_0 P) \to (M,\partial_0 P)$ be a map of pairs which restricts to the identity on $\partial_0 P$. Then 
$f$ underlies a Poincar\'e immersion if and only if there is a stable fiber homotopy equivalence $\nu_P \simeq f^\ast \nu_M $
which restricts to the identity on $\partial_0 P$.
\end{thm}

\begin{rem} The definition of Poincar\'e immersion given above is motivated by Smale-Hirsch theory. Let $P\to M$ be an immersion of smooth $n$-manifolds
with $P$ compact with spine dimension $p \le n-3$. By transversality, there is a $j > 0$ such that the composition 
$P \to M \subset M \times D^j$ is regularly homotopic to an embedding.
The normal bundle of this embedding is trivial, so a tubular neighborhood defines an embedding $P \times D^j \to M \times D^j$. Conversely, by Smale-Hirsch theory,
up to isotopy any embedding  $P \times D^j \to M \times D^j$ arises from an immersion $P\to M$.
\end{rem}

 \begin{defn}[{\cite[p.~764]{GK}}] The {\it homotopy spine dimension} of  $(P;\partial_0 P,\partial_1 P)$ (relative to $\partial_0 P$) is $\le p$ 
if
\begin{enumerate}[(i).]
\item the pair $(P,\partial_0 P)$ is homotopically $p$-dimensional (i.e., up to homotopy, $(P,\partial_0 P)$ is a retract of a CW pair of relative dimension $\le p$), and
\item the pair $(P,\partial_1 P)$ is $(d-p-1)$-connected.
\end{enumerate}
\end{defn}

\begin{rem} When $\partial_0 P = \emptyset$, the homotopy spine dimension of $P$ is $\le p$ if 
it has homotopy codimension $\ge d-p$ in the sense of \cite[p.~312]{Klein_compression}.
\end{rem}

\begin{thm}[{\cite[thm.~D]{Klein_compression2}}] \label{thm:Whitney} Assume that the homotopy spine dimension of $P$ is $\le p$ 
with $d \ge \max(p+3,2p+1,\frac{3}{2}p + 2)$. Then any Poincar\'e immersion of $P$ in $M$ (relative to $\partial_0 P$) compresses to  a Poincar\'e embedding. 
\end{thm}

 \begin{rem} When $d \ge 7$, the inequality $d \ge \max (p+3,2p+1,\frac{3}{2}p + 2)$ simplifies to  $d \ge 2p+1$.
 \end{rem}

\subsection{Surgery on a Poincar\'e embedding} Let $(P;\partial_0 P,\partial_1 P)$ be a Poincar\'e triad and suppose
 $(N;\partial_0 P,N_1)$ is a codimension one splitting.

\begin{defn} \label{defn:surgery}
By {\it surgery} on $(N;\partial_0 P, N_1)$, we mean the codimension one splitting
 \[
(N';\partial_1 P, N_1) \, ,
 \]
 i.e., $N' = \partial_1 P \cup_{\partial_{01} P} N_1$.
The {\it trace} of the surgery is the Poincar\'e cobordism from $N$ to $N'$ given by the triad
\[
((N \times I)\cup_{\partial_0 P \times\{1\}} P; N, N') \, .
\]
\end{defn}
Conversely, surgery on $(N';\partial_1 P, N_1)$ recovers $(N;\partial_0 P, N_1)$
with the trace turned ``upside down.''

More generally, suppose one is given a codimension zero Poincar\'e embedding $\partial_0 P \to N$. Then $N \simeq W$
for some codimension one splitting $(W;\partial_0 P,N_1)$. Surgery on the latter
produces another codimension one splitting $(W';\partial_1 P,N_1)$. One says in this instance $W'$ is the
 result of {\it surgery on the Poincar\'e embedding} $\partial_0 P \to N$.

\begin{ex} Let
\[
(P;\partial_0 P,\partial_1 P) = (D^{p+1} \times D^{q};S^p \times D^q, D^{p+1} \times S^{q-1})\, .
\] 
Surgery in the above sense on a codimension zero smooth embedding $\partial_0 P\to N$ results in what one classically means by a surgery on a framed
embedded $p$-sphere. 
\end{ex} 

\section{Poincar\'e normal maps}\label{sec:normal-maps}

 Let $(M^d,\xi_M,\alpha_M)$ and $(X^d,\xi,\alpha)$ be Spivak normal structures. 
In particular, $M$ and $X$ are equipped with fundamental classes.

\begin{defn} A {\it normal map} from $M$ to $X$ consists of a pair 
\[
(f,b)
\] in which 
\begin{enumerate}[(i).]
\item $f\: M^d \to X^d$ is a map of degree one (i.e., $f_\ast([M]) = [X] \in H_d(X;L)$), and 
\item $b\: \xi_M \simeq f^\ast \xi$ is a fiber homotopy equivalence (i.e.,
$b\: M \times I \to BG$ is a homotopy from $\xi_M$ to $\xi \circ f$).
\end{enumerate}
Similarly, one has the notion of a normal map of Poincar\'e pairs.

A {\it normal bordism} (also known as a {\it normal cobordism}) from $(f_0,b_0) \: M_0 \to X$ to $(f_1,b_1)\: M_1 \to X$ is
a normal map of triads 
\[
((f,b);(f_0,b_0),(f_1,b_1)) \: (W;M_0,M_1) \to (X\times I;X\times\{0\},X\times\{1\})\, .
\]
Let 
\[
\cal N(X,\xi)
\] denote the set of cobordism classes of normal maps with target $X$.
\end{defn}

\begin{rem} The fiber homotopy equivalence datum creates notational clutter.  There are three other notions of Poincar\'e normal map. All four notions become equivalent upon passing to normal cobordism, provided that the map $\xi\: X\to BG$ has been converted to a fibration. Two of the
 three notions do not require specifying a fiber homotopy equivalence. This is a feature unique to the Poincar\'e duality category.
 \end{rem}

\begin{defn}  A {\it strict normal map}
\[
f\: M \to X
\]
is a map of spaces $f\: M^d \to X^d$
such that
\begin{enumerate}[(i).]
\item  $f$ is a map over $BG$, i.e., $\xi \circ f = \xi_M$, and
\item $f$ is of degree one:  $f_\ast([M]) = [X]$.
\end{enumerate}
\end{defn}

Similarly, the notion of strict normal map extends to Poincar\'e pairs and cobordisms. Let
$\cal N'(X,\xi)$ denote the cobordism classes of strict normal maps with target $X$.

In the definition of normal map given above, it was not required that $f$ preserve the normal invariants. 
The following definition makes this modification. In the following, we
 do not need to assume  that $\xi\: X\to BG$ is a fibration.

\begin{defn} [Compare {\cite[p.~73]{HV}}]
 A {\it morphism} of Spivak normal structures 
 \[
 (f,b)\: (M,\xi_M,\alpha_M) \to (X,\xi ,\alpha)
 \]
consists of 
\begin{enumerate}[(i).]
\item a map $f\: M \to X$, and 
\item a fiber homotopy equivalence $b\: \xi_M \simeq f^\ast\xi$, such that
\[
f_\ast(\alpha_M) = \alpha\, ,
\]
where $f_\ast\: \pi_d(M^{\xi_M}) \cong \pi_d(M^{f^\ast\xi}) \to  \pi_d(X^{\xi})$ is the homomorphism
induced by $(f,b)$.
\end{enumerate}

The set of cobordism classes of morphisms of Spivak normal structures with target $(X,\xi ,\alpha)$ is denoted by
\[
\cal N(X,\xi,\alpha)\, .
\]
\end{defn}

The last notion of normal map combines the preservation of normal invariants with 
the strictness constraint:

\begin{defn}  A {\it strict morphism} of Spivak normal structures 
\[
f\: (M,\xi_M,\alpha_M) \to (X,\xi ,\alpha)
\]
is a map of spaces $f\: M \to X$
such that
\begin{enumerate}[(i).]
\item  $f$ is a map over $BG$, and
\item $f_\ast(\alpha_M) = \alpha$.
\end{enumerate}

Similarly, one has the notion of strict cobordisms of Spivak normal structures. The set of cobordism classes
of strict morphisms with target $(X,\xi ,\alpha)$ is denoted by $\cal N'(X;\xi,\alpha)$.
\end{defn}

Forgetting structure defines a commutative square
\begin{equation} \label{eqn:various-normal}
\xymatrix{
\cal N'(X,\xi,\alpha) \ar[r]\ar[d] & \cal N(X,\xi,\alpha) \ar[d] \\
\cal N'(X,\xi) \ar[r] & \cal N(X,\xi) 
}
\end{equation}

\begin{lem}[{\cite[cor.~5.12]{klein_surgery1}}]\label{cor:split-normal} 
If the map $\xi\: X\to BG$ is a fibration,  then every  map of \eqref{eqn:various-normal} is a bijection.
\end{lem}

\begin{rem} \label{rem:norma-map-convention} Unless otherwise specified, the map $\xi\: X\to BG$ is assumed to be a fibration and
a normal map refers to a representative of the set $\cal N'(X,\xi)$, i.e., a strict normal map $f\: M \to X$. The same convention will
apply to normal maps of pairs and normal cobordisms.
Similarly, when discussing morphisms of Spivak normal structures, we mean them in the strict sense.
\end{rem}

\section{Forms and formations}\label{sec:forms}

\subsection{$\epsilon$-symmetric forms} Let $\epsilon \in \{\pm 1\}$.
Let $\pi$ be a group equipped with a homomorphism $w\: \pi \to \{\pm 1\}$. 
Set $\Lambda = \Bbb Z[\pi]$. We equip $\Lambda$ with the involution $a \mapsto \bar a$ which 
on group elements is given by $g\mapsto w(g)g^{-1}$.

\begin{defn} \label{defn:symmetric-form}
An {\it $\epsilon$-symmetric form} is a pair
\[
(K,\lambda)
\]
in which $K$ is a finitely generated free $\Lambda$-module and
\[
\lambda\: K \times K \to \Lambda
\]
is a function such that for all $x,x',y,y'\in K$ and $a,b\in \Lambda$ we have
\begin{enumerate}[(i).]
\item $\lambda(x+x',y) = \lambda(x,y) + \lambda(x',y)$;
\item $\lambda(x,y+y') = \lambda (x,y) + \lambda(x,y')$;
\item $\lambda(y,x) = \epsilon\overline{\lambda(x,y)}$;
\item $\lambda(ax,by) = b\lambda(x,y)\bar a$.
\end{enumerate}
\end{defn}

The form is {\it non-singular} if the associated map of $\Lambda$-modules
\begin{align*}
\hat \lambda \: K &\to K^\ast := \hom_{\Lambda}(K,\Lambda)\\
x &\mapsto (y\mapsto \lambda(x,y))
\end{align*}
is an isomorphism; the left action on the target is defined by $(a\ell)(x) := \ell(x)\bar a$.

\subsection{$\epsilon$-quadratic forms}
Set $\epsilon = (-1)^k$ and let 
\begin{equation} \label{eqn:Qk}
Q_k := Q_k(\pi,w)
\end{equation}
denote the abelian group quotient
\[
\Lambda/\langle r - \epsilon \bar r\rangle \, ,
\]  where $\langle r - \epsilon \bar r\rangle \subset \Lambda$ is the additive subgroup generated by elements of the form $r - \epsilon \bar r$ for $r \in \Lambda$.

\begin{ex} If $\pi$ is trivial, then
\[
Q_k = \begin{cases} \Bbb Z, & k \text{ even},\\
\Bbb Z_2, & k \text{ odd}.\\
\end{cases}
\]
\end{ex}

\begin{defn} \label{defn:quadratic-form}
A {\it $\epsilon$-quadratic form} is an ordered triple
\[
(K,\lambda,\mu)
\]
such that $(K,\lambda)$ is  an $\epsilon$-symmetric form and
\[
\mu\: K \to Q_k
\]
is a function such that for all $x,y\in K$ and  $a \in \Lambda$, we have
\begin{enumerate}[(i).]
\item $\lambda(x,y) = \mu(x+y) - \mu(x) -\mu(y)$;
\item $\lambda(x,x) = \mu(x) + \epsilon \overline{\mu(x)}$;
\item $\mu(ax) = a\mu(x) \bar a$.
\end{enumerate}
In (i), we are regarding $\lambda(x,y) \in Q_k$ via the quotient homomorphism $\Lambda \to Q_k$. In (ii), the right side uses the symmetrization operator $s_{\epsilon} \: Q_k \to \Lambda$ defined by $s_{\epsilon}([r]) = r + \epsilon \bar r$.

One says that $(K,\lambda,\mu)$ is {\it non-singular} if $(K,\lambda)$ is non-singular.
\end{defn}

The definitions of symmetric and quadratic forms
extend by the same formulas to finitely generated
projective modules. We use this extension in the
hyperbolic construction and the general Lagrangian
statement below. The $L$-groups here use free modules.

\begin{ex}[Hyperbolic Forms] For a finitely generated projective module $L$, one has a
nonsingular form
\[
H_{\epsilon}(L) := (L\oplus L^\ast, \lambda,\mu)
\]
in which
$\lambda((x,f),(y,g)) = f(y) + \epsilon \overline{g(x)}$
and $\mu(x,f) = [f(x)] \in Q_k$. Note that the adjoint $\hat\lambda\: L \oplus L^* \to (L \oplus L^*)^\ast = L^* \oplus L$
corresponds to the matrix
\[
\lambda = \begin{pmatrix}
0 & 1\\
\epsilon & 0
\end{pmatrix}
\] 
(cf.~\cite[defn.~11.6(iv)]{Ranicki_ags}).
\end{ex}

\subsubsection{Cobordism of forms}
Nonsingular $\epsilon$-quadratic forms 
\[
(K,\lambda,\mu) \quad \text{ and } \quad (K', \lambda',\mu')
\] are {\it cobordant}
if there exists an isomorphism
\[
(K,\lambda,\mu) \oplus H_{\epsilon}(F) \cong (K',\lambda',\mu') \oplus H_{\epsilon}(F')
\]
for finitely generated free $\Lambda$-modules $F,F'$ (cf.~\cite[defn~11.55]{Ranicki_ags}).

\subsubsection{Lagrangians}
A {\it Lagrangian} \cite[p.~281]{Ranicki_ags} for an $\epsilon$-quadratic form $(K,\lambda,\mu)$ is a direct summand 
$L \subset K$
such that $\mu(L) = 0$ and moreover $L^\perp = L$, where 
\[
L^\perp = \{x\in K|\, \lambda(x,y)= 0, \text{ for all } y \in L\}\, .
\]
For example, $L$ is a Lagrangian for the hyperbolic form $H_{\epsilon}(L)$.

According to  \cite[thm.~11.51]{Ranicki_ags}, a non-singular $\epsilon$-quadratic form $(K,\lambda,\mu)$
admits a Lagrangian if and only if it is isomorphic to a hyperbolic form $H_{\epsilon}(L)$.
Observe that when restricted to  $L \subset H_{\epsilon}(L)$, the function $\mu$ vanishes.

\subsection{Formations}
\begin{defn}\label{defn:formation}
A  {\it ($\epsilon$-quadratic) formation}
is an ordered  5-tuple
\[
(K,\lambda,\mu;F,G)
\]
in which $(K,\lambda,\mu)$ is a non-singular $\epsilon$-quadratic form over $\Lambda$ and 
$F,G \subset K$ are Lagrangians, such that $F$ and $G$ are finitely generated and free $\Lambda$-modules.
\end{defn}

An {\it isomorphism} 
\[
(K,\lambda,\mu;F,G) @> \cong >>  (K',\lambda',\mu'; F',G')
\] of $\epsilon$-quadratic formations is 
an isomorphism $f\: (K,\lambda,\mu) \to (K',\lambda',\mu')$ of forms such that $f(F) = F'$ and $f(G) = G'$.

Ranicki \cite[prop.~12.3(ii)]{Ranicki_ags} shows that every $\epsilon$-quadratic formation is isomorphic
to 
\[
(H_{\epsilon}(F);F,\alpha(F))\, ,
\]
for some automorphism $\alpha\: H_{\epsilon}(F) \to H_{\epsilon}(F)$. 

A formation is {\it trivial} if it is isomorphic to $(H_{\epsilon}(F);F,F^*)$ for some $F$. 
A {\it stable isomorphism} of formations
\[
[f]\: (K,\lambda,\mu;F,G) \to (K',\lambda',\mu';F',G') 
\]
is an isomorphism of formations
\[
f\: (K,\lambda,\mu;F,G) \oplus T @> \cong >> (K',\lambda',\mu';F',G') \oplus T'
\]
for some trivial formations $T,T'$.

\begin{defn}[Boundary formations] \label{defn:boundary-formation}
If $(K,\lambda,\mu)$ is a $(-\epsilon)$-quadratic form,\footnote{Possibly singular.} then its {\it boundary formation} is the $\epsilon$-quadratic formation
\[
\partial(K,\lambda,\mu) :=  (H_{\epsilon}(K); K, \Gamma_{(K,\lambda)})\, ,
\]
where 
\[
 \Gamma_{(K,\lambda)} := \{(x,\hat \lambda(x)) \in K\oplus K^*| \, x\in K\}
 \]
 is the {\it graph Lagrangian} of $(K,\lambda)$ (here $\hat \lambda(x)(y) := \lambda(x,y)$).
 \end{defn}
 
The graph Lagrangian and the boundary formation 
depend only on the pair $(K,\lambda)$; they do  not depend
 on $\mu$. 
 
 \begin{lem}[Ranicki {\cite[prop.~12.19]{Ranicki_ags}}]
 \begin{enumerate}[(i).]
 \item A Lagrangian of  $H_{\epsilon}(K)$ is a  graph Lagrangian if and only 
 it is a direct complement of $K^\ast$. 
 \item An $\epsilon$-quadratic formation $(K,\lambda,\mu;F,G)$ is isomorphic to a boundary
 if and only if it has a Lagrangian $H$ which is a direct complement to both $F$ and $G$.
 \end{enumerate}
 \end{lem}
 
 \subsubsection{Cobordism of formations}
 
Two $\epsilon$-quadratic formations $(K,\lambda,
 \mu;F,G)$ and  $(K',\lambda',\mu';F',G')$ are {\it cobordant} if there exist boundary formations $B,B'$ and
  a stable isomorphism
 \[
 [f] \: (K,\lambda, \mu;F,G) \oplus B \to  (K',\lambda', \mu';F',G') \oplus B'\, .
 \]
Cobordism defines an equivalence relation \cite[prop.~12.21]{Ranicki_ags}.
 
 \begin{defn}[Ranicki {\cite[defn.~11.57,12.23]{Ranicki_ags}}] Let $k \ge 0$ be an integer.
The  $L$-group \[L_{2k}(\pi,w)\] is  the abelian group of cobordism classes of non-singular $(-1)^k$-quadratic forms.
The $L$-group \[ L_{2k+1}(\pi,w)\] is  the abelian group of cobordism classes of non-singular $(-1)^k$-quadratic formations. In each case,
addition is induced
by  direct sum.
 \end{defn}
 
 \begin{rem}  $L_{d}(\pi,w)$ is canonically isomorphic to the corresponding Wall $L$-group $L^h_{d}(\pi,w)$.
\end{rem}

\section{Cup products and Hopf invariants}  \label{sec:Hopf}
\subsection{Based $\pi$-spaces} Let $\pi$ be a group. By a {\it based $\pi$-space}, we mean a based space $A$ equipped with a left action
\[
\pi \times A \to A
\]
in which the basepoint is stabilized. A morphism of based $\pi$-spaces is
a basepoint preserving equivariant map. Let $R(\pi,\ast)$ denote the category of based $\pi$-spaces.
Then $R(\pi,\ast)$ is a model category where a map $f\: A\to B$ is a weak equivalence (fibration) iff it is
a weak equivalence (resp.~fibration)  upon application of the forgetful functor $R(\pi,\ast) \to \Top_\ast$, and a cofibration is a morphism which 
 has the left lifting property with respect to the acyclic fibrations. Alternatively, a map $A\to B$
 is a cofibration if and only if it is a retract of a map $A'\to B'$ in which $B'$ is obtained from $A'$ 
 by attaching free $\pi$-cells, where a free $\pi$-cell of dimension $j$ is $D^j \times \pi$. In particular, $A$ is cofibrant 
 if and only if it is a retract of a free based $\pi$-cell complex. Every based $\pi$-space is fibrant.
 
 We say that $A$ is
{\it finite} if it is built up from a point by attaching finitely many free $\pi$-cells.  We say that
$A$ is {\it homotopy finite} if it is weakly equivalent to a finite $\pi$-space.

For cofibrant based $\pi$-spaces $A$ and $B$,
let \[
[A,B]_\pi
\] denote the based set of homotopy classes of equivariant maps $A\to B$.
Let 
\[
\{A,B\}_\pi
\] be the abelian group of stable homotopy classes of equivariant maps $A\to B$.
An {\it equivariant stable map} is a representative of  any element of $\{A,B\}_\pi$.
If $A$ is homotopy finite, then every element of $\{A,B\}_\pi$ is represented by an (unstable) equivariant map
$\Sigma^j A\to \Sigma^j B$ for some $j$.
We let $E\: [A,B]_\pi \to \{A,B\}_\pi$ denote the stabilization map.

\subsection{The equivariant cup product} Let $A$ and $B$ be cofibrant  $\pi$-spaces, with $A$ homotopy finite.
We let $\pi$ act diagonally on the smash product $B\smsh B$.
For $\pi$-equivariant stable maps $f,g\: A\to B$, the {\it equivariant cup product}
\begin{equation} \label{eqn:cup}
f\cup g \in \{A,B\smsh B\}_\pi
\end{equation}
is defined by the composition of the equivariant stable maps $f\smsh g$ and $\Delta_A$, i.e.,
\[
f\cup g := (f \smsh g) \circ \Delta_A\, .
\]
An explicit representative can be given: Let $f\: \Sigma^j A \to \Sigma^j B$ and $g\: \Sigma^j A\to \Sigma^j B$ 
be representatives.
Then $f\cup g$ is represented by the composition 
\[
\Sigma^{j} A @> \Sigma^{j} \Delta_{A} >> \Sigma^{j} A \smsh A @>f \smsh 1_A >>  \Sigma^{j} B \smsh A = B\smsh \Sigma^j A @> 1_B\smsh g >> B \smsh \Sigma^j B 
= \Sigma^j B\smsh B \, .
\]
The equivariant cup product induces a pairing
\[
\cup \: \{A,B\}_{\pi} \times \{A,B\}_{\pi}  \to \{A,B\smsh B\}_{\pi} \, .
\]

The proof of the following result is straightforward and we leave its details to the reader.

\begin{lem} \label{lem:cup-prop} The equivariant cup product satisfies the following identities:
\begin{enumerate}[(i).]
\item $g\cup f = \tau\circ (f\cup g)$, where $\tau\: B\smsh B \to B\smsh B$ is the twist map;
\item $(f+f')\cup g = f\cup g + f'\cup g$;
\item If $C$ is a based $\pi$-space and $a,b\in \{B,C\}_{\pi}$, then $(a\circ f)\cup (b\circ g) = (a\smsh b)\circ (f\cup g)$;
\end{enumerate}
\end{lem}

\subsection{The equivariant stable Hopf invariant} For a cofibrant based $\pi$-space $B$, let 
\[
D_2(B)  := (B\smsh B) \smsh_{\Bbb Z_2} (E\Bbb Z_2)_+
\]
be the {\it quadratic construction}. It is a based $\pi$-space. The quadratic construction $D_2$ is a functor on the category of  equivariant stable maps, so
if $f\: A\to B$ is an equivariant stable map,  it induces an equivariant stable map
\[
D_2(f)\: D_2(A) \to D_2(B)\, .
\]
The {\it equivariant stable Hopf invariant}
\[
H\: \{A,B\}_{\pi} \to \{A,D_2(B)\}_{\pi}
\]
is defined as in \cite[\S4]{klein2026stablehopf}. We will not detail its construction here. Rather, our focus will be 
on reviewing its properties.

For equivariant stable maps $f,g\: A\to B$, let
\[
f\cup_2 g \in \{A,D_2(B)\}_{\pi}
\] 
be defined as  the stable composition
\[
A  @> f\cup g >> B \smsh B \to D_2(B)\, ,
\]
where the first map is the equivariant  cup product (see \eqref{eqn:cup} above) and the second map  is the equivariant homotopy class defined by the diagram
\[
B \smsh B @<\sim << (B \smsh B) \smsh (E\Bbb Z_2)_+@>>> D_2(B)\, ,
\]
in which the left map is induced by the constant map $E\Bbb Z_2\to \ast$ and the right map is the quotient map defined by 
taking coinvariants. 

Let 
\[
\tr\: D_2(B) \to B\smsh B
\] denote the equivariant stable {\it transfer map} associated with the double
cover $(B^{\times 2} \times E\Bbb Z_2, B^{\vee 2}\times E\Bbb Z_2)\to (B^{\times 2}  \times_{\Bbb Z_2} 
E\Bbb Z_2, B^{\vee 2} \times_{\Bbb Z_2} E\Bbb Z_2)$.

 Let 
 \[
 E\: [A,B]_{\pi} \to \{A,B\}_{\pi}
 \] be the stabilization map. In the following,  we write $\dim A \le s$ if $A$ is built up from a point by attaching free $\pi$-cells of 
 dimension at most $s$.

\begin{thm}[{\cite[Add.~D]{klein2026stablehopf}}]  \label{thm:Hopf-properties} The equivariant stable Hopf invariant satisfies:
\begin{enumerate}[(1).]
\item (Cartan Formula). $H(f+g) = H(f) + H(g) + f \cup_2 g$; 
\item (Transfer Formula). $\tr H(f)= (f\smsh f) \circ \Delta_A - \Delta_B \circ f$;
\item (Composition Formula). $H(g\circ f) = H(g)\circ f + D_2(g) \circ H(f)$.
\item (Vanishing on unstable classes). If $f = E(f')$, then $H(f) = 0$. 
\item (Desuspension). If $B$ is $r$-connected, $\dim A \le 3r+1$, and $H(f) = 0$,
then $f = E(f')$ for some $f'\in [A,B]_\pi$.
\end{enumerate}
\end{thm}

If $B'$ is a based $\pi$-space, we equip its suspension $\Sigma B'$ with the induced  $\pi$-action in which $\pi$ acts trivially on the suspension coordinate. 

\begin{cor} \label{cor:Hopf-properties} If $B = \Sigma B'$ is a suspension, then $\Delta_B$ is null homotopic and the transfer formula becomes
\[
\tr H(f) = f\cup f \, .
\]
\end{cor}

\begin{proof} If $B$ is a suspension, its reduced diagonal $\Delta_B$ is equivariantly null homotopic. Consequently,
\[
\tr H(f)= (f\smsh f) \circ \Delta_A  =  f\cup f\, . \qedhere
\]
\end{proof}

\section{Kernel modules} \label{sec:kernels}
Let $f\: M \to X$ be a map of connected based homotopy finite spaces. Let $\pi = \pi_1(X)$ and $\Lambda = \Bbb Z[\pi]$. Fix $w\: \pi\to \{\pm 1\}$
and define $\bar g = w(g)g^{-1}$. For a Poincar\'e normal map, $w$ is understood to
be the orientation character of the target.

Let $\tilde X\to X$ be the universal cover.
Define $\tilde M$ to be the pullback $M \times_X \tilde X$. Then $\tilde M \to M$ is a $\pi$-regular covering space and we have a $\pi$-equivariant map $\tilde f \: \tilde M \to \tilde X$.  

Let  
$C_\ast(\tilde M)$  and $C_\ast(\tilde X)$ be the singular chains on $\tilde M$ and $\tilde X$; these are chain complexes of free $\Lambda$-modules. Let $D_f$ be the algebraic mapping cone of 
the induced map $\tilde f_\ast \: C_\ast(\tilde M) \to C_\ast(\tilde X)$.

\begin{defn} The {\it  kernels} of $f$ are the $\Lambda$-modules
\begin{align*}
K_k(M) &= H_{k+1}(\tilde f\: \tilde M \to \tilde X) := H_{k+1}(D_f), \text{and} \\
K^k(M) &= H^{k+1}(\tilde f\: \tilde M \to \tilde X) := H^{k+1}(\hom_{\Lambda}(D_f,\Lambda))\, ,
\end{align*}
where  $\hom_{\Lambda}(D_f,\Lambda)$ is equipped with the structure of cochain complex of left $\Lambda$-modules
using $(a\ell)(z) := \ell(z)\bar a$ for $a\in \Lambda, \ell \in \hom_{\Lambda}(D_f,\Lambda)$.
\end{defn}

We list the fundamental properties of the kernels. For the proofs, see 
\cite[\S\S 10.3-5]{Ranicki_ags}. 

\begin{rems} \label{rems:kernels}
\begin{enumerate}[(1).]
\item Assume $f\: M \to X$ is $k$-connected, $k \ge 2$.
Then by the relative Hurewicz theorem, there is a canonical isomorphism 
\[
K_k(M) \cong  \pi_{k+1}(f)\, .
\]
In defining
$\pi_{k+1}(f)$, we 
fix a basepoint in $M$ and take 
 the set of basepoint preserving homotopy classes of commutative diagrams
\[
\xymatrix{
S^k \ar[r] \ar[d] & M \ar[d]^f \\
D^{k+1}\ar[r] & X\, .
}
\]
Note that an unbased map $\gamma\: S^k \to M$ can be treated as a based map once
a reference path from $\gamma(\ast)$ to the basepoint in $M$ is chosen.

\item Assume $k \ge 2$. Then $f\: M\to X$ is $k$-connected if and only if 
$f$ induces an isomorphism on fundamental groups and $K_i(M)=0$ for $i < k$. If $f$ is $k$-connected then the evaluation map
\[
K^k(M) @>>>  \hom_{\Lambda}(K_k(M),\Lambda) = K_k(M)^\ast
\]
is an isomorphism.
If in addition, $K^i(M) = 0$ for $i > k$, then $K_k(M)$ is a stably free $\Lambda$-module.

In particular, if $f$ is $2$-connected, then $f$ is a homotopy equivalence if and only if $K_i(M) = 0 = K^i(M)$ for all $i\ge 0$.
\end{enumerate}
\end{rems}

More generally, given a map of pairs $f\: (M,A) \to (X,B)$, one has
kernel modules
\[
K_k(M,A) := H_{k+1}(\tilde f) \quad \text{ and } \quad K^k(M,A) := H^{k+1}(\tilde f) \, ,
\]
where the displayed modules are the relative (co-)homology groups of the mapping cone $D_f$  of
the induced  chain map  $\tilde f_\ast \: C_\ast(\tilde M,\tilde A) \to C_\ast(\tilde X,\tilde B)$. These groups
fit into exact sequences
\begin{align}\label{eqn:exact-sequences} \begin{split}
\cdots \to & K_{k+1}(M,A) \to K_k(A) \to K_{k}(M) \to K_{k}(M,A)\to  \cdots   \qquad \text{ and }\\
\cdots \to &  K^{k}(M,A) \to K^{k}(M)   \to K^{k}(A) \to  K^{k+1}(M,A) \to \cdots 
\end{split}
\end{align}

\begin{lem}[{\cite[prop.~10.21, defn.~10.22]{Ranicki_ags}}]  \label{lem:ernel-props} Let $f\: M \to X$ be a normal map of  
Poincar\'e duality spaces of dimension $d$.
\begin{enumerate}[(i).]
\item There
are natural splittings of $\Lambda$-modules
\[
H_\ast(\tilde M) \cong K_\ast(M) \oplus H_\ast(\tilde X) \quad \text{ and } \quad
H^\ast(\tilde M) \cong K^\ast(M) \oplus H^\ast(\tilde X)\, ,
\]
where $H_\ast(\tilde Y) := H_\ast(C_\ast(\tilde Y))$ and  $H^\ast(\tilde Y) := H^\ast(\hom_{\Lambda}(C_\ast(\tilde Y),\Lambda))$.
Here, the right $\Lambda$-module $\hom_{\Lambda}(C_\ast(\tilde Y),\Lambda)$ is converted into a left $\Lambda$-module using the involution $\bar g =w(g)g^{-1}$.

\item  Poincar\'e duality induces an isomorphism
\[
K_i(M) \cong K^{d-i}(M), \quad i \ge 0
\]
(as the pullback of the orientation sheaf to $\tilde M$ is trivial, the orientation twist is incorporated into the involution convention above).

\item Assume $d= 2k$ and $f$ is $k$-connected.  Then the cup product on $H^k(\tilde M)$ combined with Poincar\'e duality induces the {\it intersection pairing}
\[
\lambda\: K_k(M) \times K_k(M) \to \Lambda\, ,
\]
which is a non-singular  $(-1)^k$-symmetric form with $K_k(M)$ stably free \cite[cor.~10.29(iii)]{Ranicki_ags}.
\end{enumerate}
\end{lem}

\begin{rem} In \S\ref{sec:kernel-form}, we will provide an alternative homotopy-theoretic construction of $\lambda$.
\end{rem}

\begin{cor} Suppose $f \: M^d \to X^d$ is a $(\lfloor\frac{d}{2}\rfloor + 1)$-connected normal map, $d \ge 2$. Then $f$ is a homotopy equivalence.
\end{cor}

Similarly, if $f\: (M^d,\partial M^d) \to (X^d,\partial X^d)$ is a normal map of Poincar\'e pairs, then we obtain duality isomorphisms 
\[
K^\ast(M,\partial M) \cong K_{d-\ast}(M)\quad  \text{ and } \quad K^\ast(M) \cong K_{d-\ast}(M,\partial M) \, .
\]

\section{\texorpdfstring
{The kernel form $(K_k(M),\lambda,\mu)$}
{The kernel form}}\label{sec:kernel-form}

Fix Spivak normal structures $(M^d,\xi_M,\alpha_M)$ and $(X^d,\xi,\alpha)$.
Let  
\[
f\: M^d \to X^d
\] 
be a  normal map.

\subsection{The function $\mu$} 
Assume $f\: M \to X$ is $j$-connected with $2 \le j  \le \lfloor d/2\rfloor$.
Then $x\in K_j(M)$ is represented
by a commutative diagram
\[
\xymatrix{
S^j \ar[r]^\gamma \ar[d] & M \ar[d]^f\\
D^{j+1} \ar[r] & X\, .
}
\]
Then $\gamma^\ast \xi_M \simeq  (f\circ \gamma)^\ast\xi$ extends in a preferred way to $D^{j+1}$ and therefore 
defines a trivialization. It follows that $\gamma\: S^j \to M$ is represented by a codimension zero Poincar\'e immersion
$S^j \times D^{d-j}\to M$.

\begin{defn}
The element $x\in K_j(M)$ is {\it represented by a Poincar\'e embedding} if the above Poincar\'e immersion 
is concordant to a Poincar\'e embedding.
\end{defn}

Now assume that $d = 2k$ and  $f \: M^{2k}\to X^{2k}$ is $k$-connected. The module $K_k(M)$ is finitely generated and stably free
\cite[cor.~10.29(iii)]{Ranicki_ags}; the kernel form below
is understood in this sense.
Moreover, the only possibly non-zero homology kernel is $K_k(M)$. 
The following is a non-simply connected version of  \cite[thm.~A]{Klein_sphere}. 
Recall the abelian group $Q_k = Q_k(\pi,w)$ of \eqref{eqn:Qk}. Let $\epsilon = (-1)^k$.

\begin{thm}\label{thm:Klein_sphere} Assume in addition $k \ge 4$.\footnote{This corrects the statement of \cite[thm.~A]{Klein_sphere}, where it is asserted that
 $k\ge 3$ suffices.}
   Then there is a function
 \[
\mu\:  K_k(M)  \to Q_k
 \]
 such that $\mu(x) = 0$ if and only if $x$ is represented by a codimension zero Poincar\'e embedding
 $S^k \times D^k \to M$. Moreover, $\mu$ satisfies
 \begin{enumerate}[(i).]
 \item $\mu(x+y) = \mu(x) + \mu(y) + \lambda(x,y)$,
 \item $\lambda(x,x)= s_{\epsilon}(\mu(x))$,
 \item $\mu(ax) = a\mu(x) \bar a$, for $a \in \Lambda$.
 \end{enumerate}
  \end{thm}

\begin{rem} In (i) above, $\lambda(x,y)$ is to be interpreted in $Q_k$ by taking
its image under the quotient homomorphism $\Lambda \to Q_k$.  
\end{rem}

\begin{proof}[Proof of Theorem \ref{thm:Klein_sphere}] Let $x\in K_k(M)$. Then $x$ is represented by an interior Poincar\'e immersion
$\varphi\: S^k \times D^k \to M$ (i.e., relative to $\partial_0 (S^k \times D^k) := \emptyset$). Let $(P,\partial P) = (S^k \times D^k, S^k \times S^{k-1})$.
The prescribed immersion determines a stable collapse
\[
\varphi^!\in\{M_+,P/\partial P\}.
\]
Pulling the stabilized embedding diagram back along the
universal cover gives its equivariant refinement
\[
\tilde\varphi^!\in
\{\tilde M_+,\tilde P/\partial\tilde P\}_{\pi},
\]
as constructed in \cite[\S1.3]{Klein_compression}.
Here $\tilde P=P\times_M\tilde M$ and
$\partial\tilde P=\partial P\times_M\tilde M$. By \cite[thm.~E]{Klein_compression} the Poincar\'e immersion is concordant to a Poincar\'e embedding if and only if
$\tilde\varphi^!$ lies in the image of the suspension map
\[
E\: [\tilde M_+, \tilde P/\partial \tilde P]_{\pi}\to  \{ \tilde M_+ , \tilde P/\partial \tilde P\}_{\pi}\, .
\]

As $k \ge 4$ implies $2k \le 3(k-1)+1$, the complete obstruction to desuspending
is given by the $\pi$-equivariant stable Hopf invariant 
\[
H(\tilde \varphi^!) \in \{\tilde M_+, D_2(\tilde P/\partial \tilde P)\}_\pi
\]
\cite[Add.~D]{klein2026stablehopf}.

By Lemma \ref{lem:Klein_sphere}  below, there is a canonical isomorphism
\[
\{\tilde M_+, D_2(\tilde P/\partial \tilde P)\}_\pi \cong Q_k\, .
\]
With respect to this isomorphism, we set $\mu(x) := H(\tilde \varphi^!)$. This gives the construction of $\mu$. 
The listed  properties are a direct consequence of Theorem \ref{thm:Hopf-properties}: 
Property (i) follows from the Cartan formula. Property (ii) 
follows from the transfer formula  and Corollary \ref{cor:Hopf-properties}, as $\tilde P/\partial \tilde P = \Sigma^k (S^k \times \pi)_+$ is a suspension. We leave the verification of these to the reader.
As to property  (iii), a straightforward calculation shows 
properties (i) and (ii) imply it  for all $a\in \Lambda$ provided that it holds for $a \in \pi$. We will verify the latter.

Let  $a\in \pi$.  Then $ax$ corresponds to the composition 
\[
\tilde M_+ @> \tilde \varphi^! >> \tilde P/\partial \tilde P @> \rho_a >> \tilde P/\partial \tilde P\, ,
\]
where  the second displayed self-map on
$\tilde P/\partial \tilde P = (S^k \smsh (S^k_+)) \smsh (\pi_+)$ is defined by $\rho_a(z,h) = (z,ha^{-1})$.
Since $\rho_a$ is unstable, we infer that $H(\rho_a)$ is trivial.
Consequently, the composition formula of Theorem  \ref{thm:Hopf-properties} implies
\[
\mu(ax) = H(\rho_a\circ \tilde \varphi^!) = D_2(\rho_a) \circ H(\tilde\varphi^!)\, .
\]
With respect to the isomorphism of Lemma \ref{lem:Klein_sphere},
$D_2(\rho_a) \circ H(\tilde \varphi^!) = a\mu(x) \bar a$.
\end{proof}

\begin{lem} \label{lem:Klein_sphere} With $M$ and $P$ as in the proof of Theorem \ref{thm:Klein_sphere},
there is an isomorphism 
\[
\{\tilde M_+, D_2(\tilde P/\partial\tilde P)\}_\pi \cong Q_k
\]
which is canonical up to a choice of $(L,[M])$.
\end{lem}

\begin{proof} The $\pi$-space $\tilde P/\partial \tilde P$
is identified with
\[
(S^k \times D^k \times \pi)/(S^k \times S^{k-1} \times \pi)  = (S^k_+ \smsh S^k) \smsh \pi_+\, .
\]
Hence, the inclusion $S^k \to (S^k_+ \smsh S^k)$ induces a $(3k-1)$-connected map
\[
D_2(S^k \smsh \pi_+) \to D_2(\tilde P/\partial \tilde P)\, .
\]
We infer that $D_2(\tilde P/\partial \tilde P)$
is $(2k-1)$-connected, with first non-trivial homotopy group
\[
\pi_{2k}(D_2(\tilde P/\partial \tilde P))  \cong \Bbb Z[\pi \times \pi]_{\Bbb Z_2}\, ,
\]
where the latter is an isomorphism of left $\Lambda$-modules; the $\Lambda$-module structure on the right is induced
by the diagonal action of $\pi$ on $\pi \times \pi$ defined by left translation, and 
$\Bbb Z_2$ acts on $\Bbb Z[\pi \times \pi]$ by sending a basis element $(g,h)\in \pi \times \pi$ to  $(-1)^k(h,g)$.

Elementary obstruction theory provides an isomorphism 
\[
\{\tilde M_+, D_2(\tilde P/\partial \tilde P)\}_\pi \cong H^{2k}(M;  \Bbb Z[\pi \times \pi]_{\Bbb Z_2})\, ,
\]
where the target denotes the cohomology  in degree $2k$ of $M$ with coefficients in the local coefficient module $ \Bbb Z[\pi \times \pi]_{\Bbb Z_2}$.
Poincar\'e duality provides an isomorphism
\[
\cap [M]\: H^{2k}(M;  \Bbb Z[\pi \times \pi]_{\Bbb Z_2}) \cong H_0(M; \Bbb Z^w\otimes  \Bbb Z[\pi \times \pi]_{\Bbb Z_2})\, ,
\]
where $\Bbb Z^w$ is the local coefficient right $\Lambda$-module corresponding to the orientation sheaf $L$. Its underlying abelian group
is $\Bbb Z$ and $\Lambda$ acts on the right by $1\cdot g := w(g)$ for $g\in \pi$.

By the definition of twisted homology in degree zero, we infer
\[
H_0(M; \Bbb Z^w\otimes  \Bbb Z[\pi \times \pi]_{\Bbb Z_2}) \cong \Bbb Z^w\otimes_{\Bbb Z[\pi]}  \Bbb Z[\pi \times \pi]_{\Bbb Z_2}\, .
\]
Finally,  one has an isomorphism 
\[
 Q_k @> \cong >> \Bbb Z^w\otimes_{\Bbb Z[\pi]}  \Bbb Z[\pi \times \pi]_{\Bbb Z_2}
 \]
which is induced on group elements by $g\mapsto  w(g)\otimes (g,1)$. Its inverse is given by $1 \otimes [u,v] \mapsto w(u)[v^{-1}u]$.
Assembling the above isomorphisms, we infer
\begin{align*}
\{\tilde M_+, D_2(\tilde P/\partial\tilde P)\}_\pi & \cong H^{2k}(M;  \Bbb Z[\pi \times \pi]_{\Bbb Z_2}) \, ,\\
&\cong H_0(M; \Bbb Z^w\otimes  \Bbb Z[\pi \times \pi]_{\Bbb Z_2}) \, , \\
 & \cong \Bbb Z^w\otimes_{\Bbb Z[\pi]}  \Bbb Z[\pi \times \pi]_{\Bbb Z_2} \, , \\
&\cong Q_k\, . \qedhere
\end{align*}
\end{proof}

\subsection{Alternative description of $\lambda$}  \label{subsec:another} The above argument also establishes a canonical isomorphism
\begin{equation} \label{eqn:another}
\{\tilde M_+, \tilde P/\partial\tilde P \smsh \tilde P/\partial\tilde P\}_\pi \cong \Lambda\, .
\end{equation}
With respect to this isomorphism, an alternative description
of the intersection pairing $\lambda\: K_k(M) \times K_k(M) \to \Lambda$ is  as follows:
Represent classes $x,y\in K_k(M)$ by Poincar\'e immersions $\varphi,\psi\: S^k \times D^k \to M$
and set
\[
\lambda(x,y) = \tilde \varphi^!\cup \tilde \psi^!\in \{\tilde M_+, \tilde P/\partial\tilde P \smsh \tilde P/\partial\tilde P\}_\pi \, ,
\]
where $\tilde \varphi^! \cup \tilde \psi^!$ is the equivariant cup product  of $\tilde \varphi^!$ and $\tilde \psi^!$ \eqref{eqn:cup}. By 
\eqref{eqn:another} and Lemma \ref{lem:cup-prop}, this construction defines a $(-1)^k$-symmetric form. We leave it to the reader
to show that the current description coincides with that of Lemma \ref{lem:ernel-props}.

\subsection{Making the Poincar\'e embeddings disjoint}

Throughout this subsection, $f\:M^{2k}\to X^{2k}$ is a
$k$-connected Poincar\'e normal map between connected closed
Poincar\'e spaces, and $k\ge4$. Set
$\pi=\pi_1(M)\cong\pi_1(X)$ and $\Lambda=\Bbb Z[\pi]$, and put
\[
(V,\partial V)=(S^k\times D^k,S^k\times S^{k-1}),
\qquad
(P,\partial P)=\coprod_{i=1}^r(V,\partial V).
\]

\begin{prop}\label{prop:disjoint-embedding}
Suppose that $x_1,\dots,x_r\in K_k(M)$ satisfy
\[
\mu(x_i)=0 \qquad\text{and}\qquad \lambda(x_i,x_j)=0
\]
for $1\le i,j\le r$. Then there is an interior Poincar\'e
embedding $P\to M$ whose restriction to the $i$-th summand
represents $x_i$.
\end{prop}

We first establish the following complement lemma. All mapping
spaces in its proof are derived, based, equivariant mapping spaces;
stabilization is taken in the naive equivariant sense.

\begin{lem}\label{lem:disjoint-embedding}
Let $W\to M$ be an interior Poincar\'e embedding with complement
$C$, where $W$ has homotopy spine dimension at most $k$. Let
$\psi\:V\to M$ be another interior Poincar\'e embedding. Suppose
that a homotopy is given from $\psi$ to a composite
\[
V\xrightarrow{\widehat\psi}C\xrightarrow{j}M.
\]
Then $\widehat\psi$ underlies an interior Poincar\'e embedding
with the stable normal data transported from $\psi$ along the
given homotopy.
\end{lem}

\begin{proof}
The inclusion $j\:C\to M$ is the homotopy cobase change of
$\partial W\to W$, and is therefore $(k-1)$-connected. In
particular, it induces an isomorphism on fundamental groups.
For any space over $M$, a tilde will denote its pullback along
the universal cover $\tilde M\to M$. Thus $\tilde C$
is the universal cover of $C$.

The Spivak normal fibration of $C$ is the restriction of that
of $M$. Transport the stable normal data of $\psi$ along the
chosen homotopy. By \cite[thm.~A and its addendum]{Klein_immersion},
these data are realized by a Poincar\'e immersion of
$\widehat\psi$. The homotopy also identifies the two induced
covers of $V$. Set
\[
T=\tilde V/\partial\tilde V,
\qquad A_C=\tilde C/\partial\tilde C.
\]
Since $V$ is simply connected, there is an equivariant
identification
\[
T\cong(S^k_+\smsh S^k)\smsh\pi_+.
\]
In particular, the underlying space of $T$ is $(k-1)$-connected.

Excision supplies an equivariant homotopy cofiber sequence
\[
\tilde W_+\longrightarrow \tilde M_+
\xrightarrow{c}A_C.
\]
For a based $\pi$-space $A$, write
\[
\mathcal E(A)=\operatorname{map}_*^\pi(A,T),
\qquad
\mathcal E^{\mathrm{st}}(A)=\operatorname{map}_*^\pi(A,QT),
\]
where $QT=\Omega^\infty\Sigma^\infty T$. We then have
commutative diagram whose rows are homotopy fiber sequences:
\[
\xymatrix@C=18pt@R=18pt{
\mathcal E(A_C) \ar[r]^{c^*}\ar[d] &
\mathcal E(\tilde M_+) \ar[r]\ar[d] &
\mathcal E(\tilde W_+) \ar[d] \\
\mathcal E^{\mathrm{st}}(A_C) \ar[r]_{c^*} &
\mathcal E^{\mathrm{st}}(\tilde M_+) \ar[r] &
\mathcal E^{\mathrm{st}}(\tilde W_+).
}
\]
The vertical maps are stabilization. By the Freudenthal suspension theorem, the map
$T\to QT$ is $(2k-1)$-connected. Since $\tilde W_+$ has
equivariant homotopy dimension at most $k$, the right vertical
map is $(k-1)$-connected. Comparing homotopy fibers, the left
square is $(k-2)$-Cartesian, and in particular $0$-Cartesian.

Let $a^{\mathrm{st}}\in\mathcal E^{\mathrm{st}}(A_C)$ be the
stable collapse of the prescribed immersion of $\widehat\psi$,
and let $b\in\mathcal E(\tilde M_+)$ be the collapse of
$\psi$. Compatibility of collapse with composition provides a homotopy
\[
c^*a^{\mathrm{st}}\simeq\operatorname{st}(b)\, .
\]
These data define a point of the homotopy pullback of the other
three corners of the left square. Its $0$-Cartesianness supplies
an unstable map $a\:A_C\to T$ whose stabilization represents
$a^{\mathrm{st}}$.

Apply \cite[thm.~E]{Klein_compression} to the prescribed
immersion $V\to C$, with ambient Poincar\'e pair
$(C,\partial C)$. The unstable equivariant normal invariant
$a$ yields the desired embedding. We retain $k\ge4$, so that
the codimension is at least three and the underlying
metastable inequality $3k+4\le4k$ holds.
\end{proof}

\begin{proof}[Proof of Proposition \ref{prop:disjoint-embedding}]
By Theorem \ref{thm:Klein_sphere}, each $x_i$ is individually
represented by an interior Poincar\'e embedding of $V$, with
the stable normal data determined by its chosen representative
in $\pi_{k+1}(f)$. We argue by induction on $r$. The case
$r=1$ is already established.

Suppose that the first $r-1$ classes have been represented by
an interior Poincar\'e embedding
\[
W=\coprod_{i=1}^{r-1}V_i\longrightarrow M,
\]
and let $C$ be its complement. Choose an interior Poincar\'e
embedding $\psi\:V\to M$ representing $x_r$. As above,
$C\to M$ is $(k-1)$-connected and induces an isomorphism on
fundamental groups. Relative Hurewicz on universal covers
and excision give isomorphisms
\begin{align*}
\pi_k(M,C)
&\cong\pi_k(\tilde M,\tilde C)\\
&\cong H_k(\tilde M,\tilde C)\\
&\cong H_k(\tilde W,\partial\tilde W)
\cong\Lambda^{r-1}.
\end{align*}
Here $\tilde M$ and $\tilde C$ are simply connected;
all homology groups in this display have integral coefficients.

Choose the reference paths, lifts, and orientations used to
define the intersection pairing. The image of the core map
$S^k\to V\xrightarrow{\psi}M$ under
\[
\pi_k(M)\longrightarrow\pi_k(M,C)\cong\Lambda^{r-1}
\]
has coordinates $\lambda(x_i,x_r)$, $1\le i<r$. Indeed,
the map to relative homology is induced by collapse onto
$\tilde W/\partial\tilde W$, and its coordinates
are the equivariant intersections with the cores of the
$V_i$; compare \S\ref{subsec:another}.

All these coordinates vanish. The relative homotopy exact
sequence therefore supplies a homotopy factorization of the
core map through $C$. Since $V$ deformation retracts onto its
core, we obtain a homotopy factorization
\[
\psi\simeq j\widehat\psi\:V\longrightarrow C\longrightarrow M.
\]
Transport the chosen nullhomotopy in $X$ and its stable normal
data along this homotopy. Lemma \ref{lem:disjoint-embedding}
then represents $\widehat\psi$ by an interior Poincar\'e
embedding in $C$ with these data. Its composite in $M$ still
represents $x_r$. Gluing it to the existing embedding of $W$
produces an interior Poincar\'e embedding of all $r$ summands,
completing the induction.
\end{proof}

\section{Heegaard splittings, presentations, and the kernel formation} \label{sec:Heegard-normal}

\subsection{Heegaard splittings of normal maps}
Suppose that 
\[
f\: M^{2k+1} \to X^{2k+1}
\]
is a $k$-connected Poincar\'e normal map, $k \ge 2$. We may suppose that 
\[
(M;N,D^{2k+1})\quad \text{ and } \quad (X;X_0,D^{2k+1})
\] are codimension one splittings defining top cell decompositions.
Consider the kernel $K_k(M) = \pi_{k+1}(f)$. The latter is a finitely generated $\Lambda$-module.
Let  $x_1,\dots, x_r$ be a set of generators, and set
\begin{equation}\label{eqn:P}
(P;\partial_0P,\partial_1 P) =  \mathop{\mathsmaller{\coprod}}_{i=1}^r (D^k \times D^{k+1};S^{k-1} \times D^{k+1}, D^{k} \times S^{k}) \, .
\end{equation}
Then there is a standard codimension zero embedding $\partial_0 P\to D^{2k}$, i.e., one has
$r$ disjoint framed embedded $(k-1)$-spheres. In the following this Poincar\'e embedding is held fixed.

We identify $D^{2k}$ as part of the boundary of $N$ using
\[
D^{2k} \times \{0\} \subset \partial (D^{2k} \times I) \cong \partial(D^{2k+1}) = \partial N  \, .
\]

\begin{defn}[cf.~{\cite[defn.~12.6]{Ranicki_ags}}]
By a {\it Heegaard splitting of $f$}, we mean  a degree one map of codimension one splittings
\begin{equation} \label{eqn:Heegaard}
(f;f_0,f_1)\: (M;M_0,U)  \to (X;X_0,D^{2k+1})
\end{equation}
in which
\begin{enumerate}[(i)]
\item $U$ is the trace of the surgery on $\partial_0 P \subset D^{2k}$, i.e., 
\[
U =  (D^{2k} \times I) \cup_{\partial_0 P \times \{0\}} P \cong
\natural_{i=1}^r S^k \times D^{k+1} \, 
\]
(where for convenience, we have attached the handles to  $D^{2k} \times \{0\}$), and
\item the restriction of $f_1\: U \to D^{2k+1}$ to $i$-th summand $j_i\: S^k\times D^{k+1} \subset U$
represents $x_i$  with its specified null-homotopy.
\end{enumerate}
\end{defn}

More precisely, the second condition requires that $x_i$ be represented by a commutative diagram
\[
\xymatrix{
S^k \ar[r]^(.6){j_i}  \ar[d] & U \ar[r]^{\subset} \ar[d]_{f_1} & M \ar[d]^{f} \\
D^{k+1}  \ar[r] & D^{2k+1} \ar[r]_{\subset} &  X\, .
}
\]
\begin{prop} \label{prop:Heegaard} Let $g \: M^{2k+1} \to X^{2k+1}$ be a $k$-connected Poincar\'e normal map.
Then $g$ is homotopic to a normal map $f$ which admits 
a Heegaard splitting.
\end{prop}

The proof is deferred until Appendix \ref{sec:Heegaard-appendix}.

\subsection{Presentations} Let $f\: M^{2k+1} \to X^{2k+1}$ be a $k$-connected Poincar\'e normal map.

\begin{defn}[{\cite[defn.~12.11]{Ranicki_ags}}]\label{defn:presentation} A {\it presentation} 
of $f$ is a normal cobordism
\[
(e;f,f')\: (W^{2k+2};M,M') \to (X\times I;X\times\{0\}, X\times \{1\}) 
\]
such that $e$ is $(k+1)$-connected and $f'$ is $k$-connected.
\end{defn}

A presentation obtained as the trace of Poincar\'e surgeries
on finitely many framed $k$-spheres, with chosen
nullhomotopies in $X$ and compatible normal data, will be
called a {\it handle presentation}. We stabilize handle presentations when necessary so that their distinguished kernel modules are free.

For such a presentation $(e;f,f')$, set
\[
F=K_{k+1}(W,M'),\qquad G=K_{k+1}(W).
\]
The associated formation is
\begin{equation}\label{eqn:presentation-formation}
(H_{(-1)^k}(F);F,G)\, ,
\end{equation}
where $G$ is identified with the image of its Lagrangian
inclusion
\[
(\gamma,\delta)\:
G\longrightarrow F\oplus F^\ast.
\]
The first component is the inclusion-induced homomorphism
\[
\gamma\:K_{k+1}(W)\longrightarrow K_{k+1}(W,M'), 
\]
and the second is the composition 
\[
K_{k+1}(W) \longrightarrow K_{k+1}(W,M) \cong K^{k+1}(W,M') \cong K_{k+1}(W,M')^\ast
\]
where the first isomorphism follows from Poincaré-Lefschetz duality, and the second one follows from the 
universal coefficient isomorphism for the kernel (compare~\cite[cor.~10.29(ii)]{Ranicki_ags}).

\subsection{The kernel formation}  
The convention will be that for all restrictions to pieces of a splitting, coverings
and kernel modules are obtained by pullback from the
universal cover of the ambient target $X$. 

For the chosen Heegaard splitting, $f_0$ is
$k$-connected. The inclusion $U \to M$ induces a surjective homomorphism
\[
K_k(U)\cong\Lambda^r\longrightarrow K_k(M)\, ,
\]
since it sends the standard basis to the chosen generators
$x_1,\dots,x_r$.
Excision and the kernel exact sequence yields
$K_i(M_0,\partial U)=0$ for $i\le k$.
By duality,  $K^q(M_0)=0$ for $q>k$.
Then the argument of
\cite[lem.~10.26(ii)]{Ranicki_ags} shows that
$K_k(M_0)$ is stably free. Consequently, so is
\[
K_{k+1}(M_0,\partial U)
\cong K^k(M_0)\cong K_k(M_0)^\ast.
\]
The other distinguished module
$K_{k+1}(U,\partial U)\cong\Lambda^r$ is free.

When necessary, we stabilize the chosen splitting
by adjoining zero generators, so that both
modules are free.

\begin{defn}[cf.~{\cite[defn.~12.7]{Ranicki_ags}}] Assume $ k \ge 2$.
Let $f\: M^{2k+1} \to X$ be a Poincar\'e normal map equipped with a Heegaard splitting $(f;f_0,f_1)$
as in \eqref{eqn:Heegaard}. 

 The {\it kernel formation} of $f$ 
is the $(-1)^k$-quadratic formation
\[
(K,\lambda,\mu;F,G) := (K_k(\partial U),\lambda,\mu; K_{k+1}(U,\partial U), K_{k+1}(M_0,\partial U))\, ,
\]
with $(K_k(\partial U),\lambda,\mu) = H_{(-1)^k}(K_{k+1}(U,\partial U))$ the hyperbolic $(-1)^k$-quadratic form over $\Lambda$ of
the normal map $f_{|\partial U}\: \partial U \to S^{2k}$ over $X$.
\end{defn}

\begin{lem}\label{lem:presentation-kernel}
Let $f\:M^{2k+1}\to X$ be a $k$-connected Poincar\'e normal
map equipped with a Heegaard splitting, and assume $k \ge 4$.
Write
\[
\cal K_f=(A,\lambda,\mu;F,G)
\]
for its kernel formation. Given a quadratic Lagrangian
$L\subset A$ complementary to $F$, there is a handle
presentation
\[
(e;f,f')\:(V;M,M')\rightarrow
(X\times I;X\times\{0\},X\times\{1\})
\]
whose associated formation is $(H_{(-1)^k}(F_V);F_V,G_V)$ as in \eqref{eqn:presentation-formation}
with 
\[
F_V = K_{k+1}(V,M')\quad \text{ and } \quad G_V = K_{k+1}(V)\, .
\] This formation is identified with $\cal K_f$
by an isometry carrying $F_V^\ast$ to $L$.

In particular, under these identifications, $\gamma_V$
coincides with
\[
p_F^L|_G\:G\longrightarrow F,
\]
where $p_F^L\:A=F\oplus L\to F$ is projection along $L$.
\end{lem}

\begin{proof}
The surgery trace of the given Heegaard splitting supplies
an initial handle presentation. Put $\epsilon=(-1)^k$ and
identify its hyperbolic form with $H_\epsilon(F)$. Write
\[
(\gamma,\delta)\:G\longrightarrow F\oplus F^\ast
\]
for its Lagrangian inclusion.

Since $L$ is complementary to $F$, it has the form
\[
L=\{(sy,y)\mid y\in F^\ast\}
\]
for a homomorphism $s\:F^\ast\to F$. Since $L$ is a quadratic Lagrangian, the restrictions of
$\lambda$ and $\mu$ to $L$ vanish.  In terms of the graph
homomorphism $s$, these conditions say that $s$ lies in the
image of $1-\epsilon T$; equivalently,
\[
s=\psi-\epsilon\psi^\ast
\]
for some homomorphism $\psi\:F^\ast\to F$. (cf.~\cite[pp.~60--61]{Wall} and
\cite[proof of prop.~12.10]{Ranicki_ags}).
Consequently,
\[
T_{-s}=
\begin{pmatrix}1&-s\\0&1\end{pmatrix}
\]
is a quadratic isometry carrying $L$ to $F^\ast$.

We realize this isometry by changing the attaching
embeddings. Let $U\to M$ be the standard handlebody in the given
Heegaard splitting, and let $C=M_0$ be its Poincar\'e
complement in $M$. Hence,
\[
M\simeq U\cup_{\partial U}C\, ,
\qquad \partial C=\partial U\, .
\]
Then $\pi_1(C) \to \pi_1(M)$ is an isomorphism. Choose finitely many
paths in $C$, with endpoints on $\partial U$, representing
the group elements occurring in a matrix for $\psi$.
Theorems \ref{thm:Smale-Hirsch} and \ref{thm:Whitney}
thicken these paths, relative to their endpoints, to
disjoint Poincar\'e embedded $1$-handles. Their union
with $U$ is a smooth handlebody $N$ Poincar\'e embedded
in $M$. Here the relative spine dimension is $1$, so
the compression inequality is satisfied.

Inside $N$, perform the framed linking and self-linking
changes of \cite[proof of thm.~5.8]{Wall}, choosing
their signed group-ring coefficients to realize $-s$.\footnote{For a related homotopy theoretic construction, compare
 with the proof of \cite[thm.~11.1]{klein_surgery1}.}
These changes are fixed near $\partial N$; the Poincar\'e
complement of $N$ is left unchanged. Transport the
nullhomotopies in $X$ and the stable normal data along
the homotopies.

The homotopies and excision identify the old and new
kernel modules. The meridians are unchanged, while the
longitudes acquire the meridian terms prescribed by
$-s$. Consequently, the new Lagrangian inclusion is
\[
(\gamma-s\delta,\delta);
\]
this is the local calculation of
\cite[pp.~60--61]{Wall}. Compare also with 
\cite[proof of prop.~12.10, p.~309]{Ranicki_ags}.
All intersections used in that calculation occur in
$N\times I$.

The new formation is therefore $T_{-s}\cal K_f$.
The inverse isometry identifies it with $\cal K_f$
and carries its standard complementary Lagrangian
$F^\ast$ to $L$. Its first component is
$\gamma-s\delta=p_F^L|_G$, as required.
\end{proof}

\begin{rem}
The above argument requires no manifold structure on the
ambient Poincar\'e spaces. Smooth handle models are used
only for local changes of presentations, supported in
Poincar\'e embedded smooth handlebodies whose Poincar\'e
complements are left unchanged. \end{rem}

\section{The surgery obstruction of a Poincar\'e normal map} \label{sec:surgery-obstruction}

Let $(f,b)\: M^d \to X^d$ be a normal map of Poincar\'e duality spaces of dimension $d\ge 8$. 
Assume that $f$ is $\lfloor \frac{d}{2}\rfloor$-connected. Let $\pi = \pi_1(X)$ and let $w\: \pi \to \{\pm 1\}$ denote the orientation character.
If $d= 2k$, we have the  kernel form 
\[
(K_k(M),\lambda,\mu) 
\]
of \S\ref{sec:kernel-form}  and if $d= 2k+1$ we have the kernel formation
\[
(K_k(\partial U),\lambda,\mu; K_{k+1}(U,\partial U), K_{k+1}(M_0,\partial U))
\] of \S\ref{sec:Heegard-normal} associated with a  Heegaard splitting 
\[
(f;f_0,f_1) \: (M;M_0,U) \to (X;X_0,D^{2k+1})\, .
\]

In the even-dimensional case, $K_k(M)$ is a finitely
generated stably free $\Lambda$-module
\cite[cor.~10.29(iii)]{Ranicki_ags}, rather than necessarily
a free module. The definitions of symmetric and quadratic
forms extend to such modules. To obtain a representative
of the surgery obstruction in the free-module formulation
of $L_{2k}(\pi,w)$, we first replace the kernel form by
\[
(K_k(M),\lambda,\mu)\oplus H_{(-1)^k}(\Lambda^r),
\]
where $r$ is sufficiently large that the underlying module
is free. The resulting class is independent of this choice,
since any two such stabilizations become isomorphic after
adding further hyperbolic summands.

\begin{defn}  The {\it surgery obstruction}
\[
\sigma(f) \in L_d(\pi,w)
\]
is the cobordism class of
\[
\begin{cases}
(K_k(M),\lambda,\mu)\, , \quad & d= 2k\, ,\\
(K_k(\partial U),\lambda,\mu; K_{k+1}(U,\partial U), K_{k+1}(M_0,\partial U))\, , & d= 2k+1\, .
\end{cases}
\]
\end{defn}

Assume next that $f\: M^d \to X^d$ is not necessarily $\lfloor \frac{d}{2}\rfloor$-connected. 
By \cite[thm.~8.9]{klein_surgery1} we may find a normal cobordism from $f$  to a $\lfloor \frac{d}{2}\rfloor$-connected normal map $f'\: M' \to X$.
In this case we define $\sigma(f) := \sigma(f')$.

\begin{prop} The surgery obstruction of $f$ is well-defined. 
\end{prop}

\begin{proof} The different ways of making $f$ a $k$-connected map by a normal cobordism do not change its surgery obstruction
 \cite[prop.~6.6(i)]{Ranicki}. If $d=2k+1$, then the cobordism class of the kernel formation
 is independent of the choice of Heegaard splitting by Proposition 
 \ref{prop:Heegaard-unique} below.
 \end{proof}

\section{Realization and the Proof of Theorem \ref{bigthm:sigma-onto}} \label{sec:realization}

In this section we prove a Poincar\'e space analog of the 
Wall realization theorem in odd dimensions \cite[thm.~6.5]{Wall}. 
We then apply this result to prove Theorem  \ref{bigthm:sigma-onto}.

Let $X$ be a connected Poincar\'e space of dimension $2k$ with $k  \ge 4$. Set $\pi := \pi_1(X)$
and let $w\: \pi \to \{\pm 1\}$ be the orientation character.
Let 
\[
A = (K,\lambda,\mu;F,G)
\]
be a $(-1)^k$-quadratic formation.

We use the following Poincar\'e space analogue of
\cite[defn.~12.16]{Ranicki_ags}:
Let
\[
(e;\id,e_2)\:(W^{2k+1};X,N_2)\longrightarrow
(X\times[0,2];X\times\{0\},X\times\{2\})
\]
be a $k$-connected Poincar\'e normal cobordism, where
$e_2$ is a homotopy equivalence. 

For the construction below, we consider a decomposition
\[
W=W_-\cup_Z W_+
\]
into two successive surgery traces. Here $W_-$ is the
trace of trivial surgeries on framed $(k-1)$-spheres
in $X$, with outgoing boundary $Z$, and $W_+$ is the
trace of surgeries on framed $k$-spheres in $Z$,
with outgoing boundary $N_2$. All kernel modules are
computed for the corresponding restrictions of $e$,
using the covering spaces pulled back from the
universal cover of $X$.

Set
\[
F_-=K_{k+1}(W_-,Z),
\qquad
F_+=K_{k+1}(W_+,Z).
\]
The boundary homomorphisms identify these modules
with submodules of $K_k(Z)$: they are injective
because $K_{k+1}(W_-)=K_{k+1}(W_+)=0$, as follows
from the handle descriptions viewed from the
respective outer boundaries. In the construction
below their images are free quadratic Lagrangians.
The relative kernel formation is
\[
(K_k(Z),\lambda_Z,\mu_Z;F_-,F_+)\, ,
\]
where we equip  $Z$ with the orientation sheaf and fundamental class as the outgoing boundary of
$W_-$. Its cobordism class is the surgery obstruction
relative to the specified outer-boundary homotopy
equivalences.

In the proof we use the notation 
\[
W_- = M\, , \quad Z = M_1\, , \quad W_+ = M'\, .
\]

\begin{thm}[Realization]\label{thm:Wall_realization} 
With respect to these assumptions, there is a Poincar\'e cobordism
$(W,\partial W)$  of dimension $2k+1$
with $\partial W = \partial_0 W\amalg \partial_1 W$,
$\partial_0 W = X$, and a normal cobordism
\[
(f;f_0,f_1)\: (W,\partial_0 W,\partial_1 W)  \to (X \times I,X\times \{0\},X\times \{1\})
\]
such that $f$ is $k$-connected, $\partial_0 W \to X\times \{0\}$ is the identity and $\partial_1 W \to X \times \{1\}$ is a homotopy equivalence. Moreover, with respect to a suitable choice of Heegaard splitting,
the kernel formation of $(f;f_0,f_1)$ is isomorphic to $A$.
\end{thm}

\begin{proof} The proof parallels \cite[prop.~12.17]{Ranicki_ags}.
By \cite[thm.~11.51]{Ranicki_ags} we may assume $(K,\lambda,\mu) = H_{(-1)^k}(F)$. Let $r$ be the rank of the free module $F$. Then 
\[
F \cong \Bbb Z[\pi]^r \cong G\, .
\]
We start by performing surgery on the identity map of $X$ to kill $r$ trivially Poincar\'e embedded
framed $(k-1)$-spheres (this is possible since we can embed the disjoint union of the spheres
trivially in $D^{2k} \subset X$). The result is the connected sum
\[
X' = X \sharp  \underbrace{(S^k \times S^k) \sharp \cdots \sharp  (S^k \times S^k)}_{r\text{ times}} \, .
\]
The trace of the surgery is given by a normal map
\[
(f';f'_0,f'_1)\: (M,M_0,M_1) \to (X\times I, X\times \{0\},X\times \{1\})
\]
in which 
\begin{enumerate}[(i).]
\item $M = (X \times I) \cup_{\partial_0 P\times \{1\}}  P$,
where $(P,\partial_0 P)$ is as in \eqref{eqn:P},
\item The restriction $f'_{|X\times I}$ is the identity and
is extended to $P$ using a null homotopy of $\partial_0 P \to D^{2k}$,
\item $f'_0\: M_0 \to X$ is the identity, and
\item  up to homotopy, $M_1 \to X$ is identified with the map $X' \to X$ which is the identity
on $X_0$ and collapses the complement of $X_0$ in $X'$ to $D^{2k}$.
\end{enumerate}
Choose the standard hyperbolic identification
\[
(K_k(M_1),\lambda,\mu)\cong H_{(-1)^k}(F)
\]
which carries
\[
\operatorname{im}\bigl(
K_{k+1}(M,M_1)\to K_k(M_1)\bigr)
\]
to the first summand $F$.

The  Lagrangian $G\subset (K_k(M_1),\lambda,\mu) = H_{(-1)^k}(F)$ can be realized
as the normal map
\[
(f'',f''_1,f''_2) \: (M', M_1,M_2) \to (X\times [1,2],X\times \{1\},X\times\{2\})
\]
that arises as the trace of a Poincar\'e surgery
of a Poincar\'e embedding  of the disjoint union of  $r$ copies of $S^k \times D^k$,
i.e., 
\[
M' = (M_1 \times [1,2]) \cup_{\partial_+ Q\times \{2\}} Q\, ,
\]
where 
\[
(Q,\partial_+ Q)  =  \mathop{\mathsmaller{\coprod}}_{i=1}^r (D^{k+1} \times D^k, S^k \times D^{k})
\]
and
\[
G = \text{image}(K_{k+1}(M',M_1) \to K_{k}(M_1)) \subset K_{k}(M_1)\, .
\]
The Poincar\'e embeddings are supplied by
Proposition \ref{prop:disjoint-embedding} as applied to
a basis of the free Lagrangian $G$;  its basis elements
have vanishing quadratic and pairwise intersection
invariants.

The kernel exact sequences of the surgery trace,
together with Poincar\'e duality, yield an exact sequence
\[
0\longrightarrow K_k(M_2)
\longrightarrow K_k(M_1)/G
\xrightarrow{\ell}G^\ast
\longrightarrow K_{k-1}(M_2)
\longrightarrow0,
\]
where $\ell([z])(g)=\lambda(z,g)$. Since the kernel
form is nonsingular and $G$ is a direct summand,
$\ell$ is surjective. We infer that $K_{k-1}(M_2)=0$ and
\[
K_k(M_2)\cong\ker(\ell)=G^\perp/G=0,
\]
with the last equality holding because $G$ is a Lagrangian.
Moreover, connectivity is retained and all remaining
kernel modules vanish by Poincar\'e duality.
Consequently, $f''_2\:M_2\to X$ is a homotopy equivalence.

Then up to identifying $[0,2]$ with $I$,  the desired normal cobordism is
\[
(f,f'_0,f_2'') \: (M\cup_{M_1} M'; M_0,M_2) \to (X\times [0,2]; X\times \{0\}, X\times \{2\})\, ,
\]
where $f = f' \cup f''$. We leave the  verification that $f$ is $k$-connected to the reader.

The two boundary images are therefore the prescribed
$F$ and $G$. Consequently, the relative kernel formation is
$(H_{(-1)^k}(F);F,G)\cong A$, as required.
\end{proof}

\begin{cor} Let $k \ge 4$. Then given $\sigma\in L_{2k+1}(\pi,w)$,  there is Poincar\'e normal cobordism
$(f,f_0,f_1)$ with $f_0$  the identity, $f_1$ a homotopy equivalence, and relative surgery obstruction  $\sigma$.
\end{cor}

\begin{proof}[Proof of Theorem \ref{bigthm:sigma-onto}]
As pointed out in the introduction, the case $d = 2k$ was already proved in \cite{klein_surgery1}. 
Let $d=2k+1$. The proof is in essence the same as that of Theorem B of \cite{klein_surgery1},
the only difference is that here we use Theorem \ref{thm:Wall_realization}  in place of  \cite[thm.~11.1]{klein_surgery1}.
For the benefit of the reader, we reproduce the argument below.

Let $j\: Y\to X$ be a 2-skeleton, i.e., a $2$-connected map from a finite complex of dimension $\le 2$.
Then by \cite[thm.~A]{Klein_haef}, there is a codimension  $d-2$ Poincar\'e embedding of $j$, i.e., there is a
homotopy pushout square
\[
\xymatrix{
\partial \bar Y \ar[r] \ar[d] & X_0\ar[d] \\
\bar Y \ar[r] & X
}
\]
and a homotopy equivalence $Y @> \simeq >> \bar Y$ such that the composition $Y \to \bar Y\to X$ is $j$.
Here, $(\bar Y,\partial \bar Y)$ and $(X_0,\partial \bar Y)$ are Poincar\'e pairs of dimension $d$ whose fundamental classes
 are compatible with the fundamental class of $X$. Furthermore, the inclusion $\partial \bar Y\to \bar Y$ is $(d-3)$-connected.
As a consequence,  $X$ is identified with the double mapping cylinder
\[
\bar Y\cup_{\partial \bar Y \times \{0\}} (\partial \bar Y \times I)\cup_{\partial \bar Y\times \{1\}} X_0\, .
\]
 By Theorem \ref{thm:Wall_realization},
if $\tau \in L_d(\pi,w) $ is an element, there is a Poincar\'e pair
$(W,\partial W)$  of dimension $d$ with $\partial W = \partial_0 W\amalg \partial_1 W$,
and a normal map 
\[
F\: (W,\partial_0 W,\partial_1 W)  \to (\partial \bar Y \times [0,1],\partial \bar Y\times \{0\},\partial \bar Y\times \{1\})
\]
such that $\partial_0 W \to \partial \bar Y\times \{0\}$ is the identity,
$\partial_1 W \to \partial \bar Y \times \{1\}
$ is a homotopy equivalence, and the surgery obstruction
of $F$ rel boundary is  $\tau$.
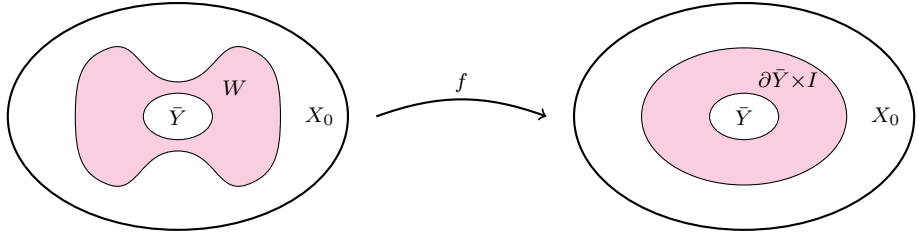
\begin{figure} 
\centering
\begin{tikzpicture}[scale=.75]
    \draw[thick] (0,0) ellipse (3cm and 2cm);
    
    \draw[thick, black] plot[smooth cycle, tension=1,scale=1.2] coordinates {(0,0.5) (1,1) (1.5,0) (1,-1) (0,-0.5) (-1,-1) (-1.5,0) (-1,1)};
       \fill[white!60!magenta!60] plot[smooth cycle, tension=1,scale=1.2] coordinates {(0,0.5) (1,1) (1.5,0) (1,-1) (0,-0.5) (-1,-1) (-1.5,0) (-1,1)};
    \draw[thick, scale=.2] (0,0) ellipse (3cm and 2cm);
    \fill[white, scale=.2](0,0) ellipse (3cm and 2cm);
    \draw[thick,->] (3.5,0) to[bend left=20] (6.5,0);
    \node at (5,.6) {{\tiny $f$}};
     \node at (0,0) {{\tiny $\bar Y$}}; 
      \node at (1,.5) {{\tiny $W$}};
         \node at (2.5,0) {{\tiny $X_0$}};
\end{tikzpicture} \hskip .2cm
\begin{tikzpicture}[scale=.75]
    \draw[thick] (0,0) ellipse (3cm and 2cm);
         \draw[thick, scale =.6] (0,0) ellipse (3cm and 2cm);
           \fill[white!60!magenta!60, scale =.6] (0,0) ellipse (3cm and 2cm);
        \draw[thick, scale =.2] (0,0) ellipse (3cm and 2cm);
         \fill[white,scale =.2](0,0) ellipse (3cm and 2cm);
     \node at (0,0) {{\tiny $\bar Y$}}; 
      \node at (.78,.65) {{\tiny $\partial \bar Y {\times} I$}};
         \node at (2.5,0) {{\tiny $X_0$}};
\end{tikzpicture}
\caption{A depiction of the Poincar\'e normal map  constructed in the proof of Theorem \ref{bigthm:sigma-onto}. 
The white region $\bar Y \amalg X_0$ is mapped using the identity. On the magenta region,  the map
supplied by Theorem \ref{thm:Wall_realization} is used.
}
\label{fig:schematicF}
\end{figure}

Let $f\: M \to X$ be defined as 
\[
 M := \bar Y\cup_{F_{|\partial_0 W}} W \cup_{F_{|\partial_1 W}} X_0  @> \text{id}_{\bar Y} \cup F \cup \text{id}_{X_0} >>
\bar Y\cup_{\partial \bar Y \times \{0\}} (\partial \bar Y \times I)\cup_{\partial \bar Y\times \{1\}} X_0 = X\, .
 \]
Then $f$  is a Poincar\'e normal map with surgery obstruction $\tau$  (cf.~\!fig.~\!\ref{fig:schematicF}).
\end{proof}


\section{\texorpdfstring
{Proof of Theorem \ref{bigthm:fund-thm} when $d=2k$}
{Proof of Theorem A in even dimensions}}
\label{sec:d=2k}

In \S\ref{sec:kernel-form} we constructed the kernel form $(K_k(M),\lambda,\mu)$ 
of a $k$-connected Poincar\'e normal map of dimension $d= 2k$ homotopy theoretically. We now use
it to prove Theorem \ref{bigthm:fund-thm} when $d=2k$. The argument is an adaptation
of the proof of \cite[prop.~11.73]{Ranicki_ags} to the Poincar\'e setting.

\begin{proof}[Proof of Theorem \ref{bigthm:fund-thm} when $d=2k$]
Let $f\: M^{2k} \to X^{2k}$ be a Poincar\'e normal map.
By \cite[thm.~8.9]{klein_surgery1}, there
is a normal cobordism of $f$ to a $k$-connected map with the same surgery obstruction. We may therefore assume that $f$ is $k$-connected.

We may also assume that $M$ and $X$ are equipped with codimension one splittings $(M;M_0,D^{2k})$ and $(X;X_0,D^{2k})$. By Proposition \ref{prop:normal-map-top-cell},
we may further assume that $f$ arises from a map of triads
\[
(f;f_0,1) \: (M;M_0,D^{2k}) \to (X;X_0,D^{2k})
\]
(compare \cite[p.~296 top]{Ranicki_ags}).

Then
\[
\sigma(f) = (K_{k}(M),\lambda,\mu) \in L_{2k}(\pi,w)\, .
\]
If $\sigma(f) = 0$, then there is an isomorphism
\[
(K_{k}(M),\lambda,\mu) \oplus H_{(-1)^k}(\Lambda^r) \cong H_{(-1)^k}(\Lambda^s) 
\]
for some non-negative integers $r,s$. Let $(P,\partial_0 P)$ be as in \eqref{eqn:P}.
Choose a (standard) interior Poincar\'e embedding
\[
\partial_0 P\to D^{2k}\, .
\]
Then the composition
\[
\partial_0 P \to D^{2k} \to M
\]
is a Poincar\'e embedding. 
Surgery on the latter produces a normal cobordism from $f$ to 
a map
\[
g \: M_1  := M \,\, \sharp \,\, (\sharp_r S^k \times S^k) @> f\sharp \text{std} >> X \sharp S^{2k}  = X
\]
with surgery kernel
\[
(K_k(M_1),\lambda_1,\mu_1) \cong H_{(-1)^k}(\Lambda^{s})\, .
\]
Apply   Proposition \ref{prop:disjoint-embedding}  to the basis elements $x_1,x_2,\dots,x_s \in K_k(M_1)$ for the Lagrangian
to represent them by disjoint Poincar\'e embeddings. Then surgery on the $x_i$
yields a homotopy equivalence $h \: M'@> \simeq >> X$.
\end{proof}

\section{\texorpdfstring
{Proof of Theorem \ref{bigthm:fund-thm} when $d=2k+1$}
{Proof of Theorem A in odd dimensions}}
\label{sec:d=2k+1}

Let $f\: M^{2k+1} \to X^{2k+1}$ be a Poincar\'e normal map, $k \ge 4$.

\begin{proof}[Proof of Theorem \ref{bigthm:fund-thm}
when $d=2k+1$]
By \cite[thm.~8.9]{klein_surgery1},
we may assume that $f$ is $k$-connected.
By Proposition \ref{prop:Heegaard}, we can assume $f$ is equipped with a Heegaard splitting
\[
(f;f_0,f_1)\: (M;M_0,U) \to (X; X_0,D^{2k+1})\, .
\]
Let $\cal K_f$ be its kernel formation.

In \S\ref{sec:Heegard-normal} we associated to $(f;f_0,f_1)$ the kernel formation
\[
(K,\lambda,\mu;F,G) := (K_k(\partial U),\lambda,\mu;K_{k+1}(U,\partial U), K_{k+1}(M_0,\partial U))\, .
\]
Since $\sigma(f)=0$, the purely algebraic criterion
\cite[rem.~12.31]{Ranicki_ags} implies that $\cal K_f$
is stably isomorphic to a boundary formation.
Proposition \ref{prop:operations} below shows that $f$
is normally cobordant to a homotopy equivalence.
\end{proof}

\begin{prop}\label{prop:operations}
Let $f\:M^{2k+1}\to X$ be a $k$-connected Poincar\'e
normal map equipped with a Heegaard splitting, where
$k\ge4$. If its kernel formation is stably isomorphic
to a boundary formation, then $f$ is normally cobordant
to a homotopy equivalence.
\end{prop}

\begin{proof}
Put $\epsilon=(-1)^k$ and let
\[
T_r=(H_\epsilon(\Lambda^r);
          \Lambda^r,(\Lambda^r)^\ast)
\]
denote the standard trivial formation. Suppose
\[
\cal K_f\oplus T_r\cong\partial q\oplus T_s.
\]
Since
\[
T_s\oplus T_s\cong
\partial H_{-\epsilon}(\Lambda^s),
\]
we obtain
\[
\cal K_f\oplus T_{r+s}
\cong
\partial\bigl(q\oplus H_{-\epsilon}(\Lambda^s)\bigr).
\]
Adjoining standard unknotted zero generators to the
Heegaard splitting realizes the stabilization on the
left without changing $f$
\cite[\S6, operation~(T1)]{Wall}.
We may therefore assume that $\cal K_f$ is isomorphic
to a boundary formation.

Choose a quadratic Lagrangian $L$ complementary to
both distinguished Lagrangians $F$ and $G$
\cite[prop.~12.19(ii)]{Ranicki_ags}.
Apply Lemma \ref{lem:presentation-kernel} with this $L$.
The resulting presentation has
\[
\gamma_V=p_F^L|_G,
\]
which is an isomorphism because $G$ is complementary
to $L$. The kernel exact sequence
\[
0\longrightarrow K_{k+1}(M')
\longrightarrow K_{k+1}(V)
\xrightarrow{\gamma_V} K_{k+1}(V,M')
\longrightarrow K_k(M')
\longrightarrow0
\]
gives
\[
K_{k+1}(M')=K_k(M')=0.
\]
All remaining kernel modules vanish by connectivity
and Poincar\'e duality. Thus $f'\:M'\to X$ is a
homotopy equivalence.
\end{proof}


\section{Proof of Theorem \ref{bigthm:stronger-exactness}} \label{sec:transversality}

The proof of Theorem \ref{bigthm:stronger-exactness} relies crucially on the results of \cite{klein_surgery1}.

Let $X$ be a connected finite type space equipped with stable spherical fibration $\xi$. Let $\pi = \pi_1(X)$ and $w = w_1(\xi)$.
In \cite[\S3.4]{klein_surgery1}, $\cal F_d(X,\xi)$ was defined as a cobordism group
whose generators consist of $(P^d,\phi,\alpha)$, where
\begin{enumerate}[(i).]
\item $P^d$ is a Poincar\'e duality space
of dimension $d$,
\item $\phi\: P^d \to X$ is a map, and
\item $\alpha$ is the homotopy class of a commutative diagram
\[
\xymatrix{
S^d \ar[r]\ar[d] & P^{\phi^\ast \xi} \ar[d] \\
D^{d+1} \ar[r] &X^\xi
} 
\]
such that the top horizontal map provides a normal invariant for $(P,\phi^\ast \xi)$.
\end{enumerate}

\begin{rem} \label{rem:trivial} When $(X,P)$ is a Poincar\'e pair of dimension $d+1$ with normal invariant
 $\alpha$, the associated generator defines the trivial element
of $\cal F_d(X,\xi)$ (cf. \cite[\S3.4, prop.~3.4]{klein_surgery1}).
\end{rem}

In \cite[\S4]{klein_surgery1} we defined a natural homomorphism
\[
\sigma' \: \cal F_d(X;\xi) \to L_d(\pi,w)\, .
\]

\begin{proof}[Proof of Theorem  \ref{bigthm:stronger-exactness}] Assume $d \ge 8$. Let $Y\to X$ be a finite CW 2-skeleton.
Then one has a commutative diagram
\[
\xymatrix{
\cal F_d(Y,\xi) \ar[r]^{\sigma'} \ar[d]_{\cong} & L_d(\pi,w) \ar[d]^{=}\\
\cal F_d(X,\xi) \ar[r]_{\sigma'}  & L_d(\pi,w)\, .
}
\]
By \cite[cor.~7.3]{klein_surgery1}, the left vertical map is an isomorphism. It therefore suffices to prove that
the upper horizontal map is an isomorphism.

In  \cite[\S12]{klein_surgery1}, we constructed a Poincar\'e $(d+1)$-thickening of $Y$, i.e., a $2$-connected Poincar\'e pair $(\bar Y,\partial \bar Y)$ of dimension $(d+1)$
equipped with a homotopy equivalence $Y \simeq \bar Y$. Moreover, the Spivak fibration of 
$(\bar Y,\partial \bar Y)$ coincides with the pullback of $\xi$. 
We also constructed a function $
t\: \cal N(\partial \bar Y,\xi) \to \cal F_d(Y,\xi) 
$
such that the composition
\[
\cal N(\partial \bar Y,\xi) @> t >> \cal F_d(Y,\xi)  @> \sigma' >> L_d(\pi,w)
\]
is the surgery obstruction $\sigma$. 
By Theorem \ref{bigthm:sigma-onto} it follows that $\sigma'$ is onto. It remains to show that $\sigma'$
is one-to-one.

By \cite[cor.~6.11]{klein_surgery1}, the function $t$ is onto, since $\dim Y \le 2$. Let $x\in \cal F_d(Y,\xi)$ be an element
such that $\sigma'(x) = 0$. Then $x$ lifts to a normal map $f\: M \to \partial \bar Y$ and $\sigma(f) = \sigma'(x) = 0$. By Theorem
\ref{bigthm:fund-thm}, it follows that $f$ is normally cobordant to a homotopy equivalence $f'\: M' \to \partial \bar Y$.
It follows from \cite[cor.~6.11,prop.~6.5]{klein_surgery1}, that $t(f')=x$ is represented by the Poincar\'e pair $(\bar Y,\partial \bar Y)$.
We infer that $x$ is trivial. Consequently, $\sigma'$ is one-to-one.
\end{proof}

\section{Proof of Theorem \ref{bigthm:handlebody}} \label{sec:handlebody}

The scheme of the proof of Theorem \ref{bigthm:handlebody} is by an induction
similar to that of \cite[thm.~6.1]{HV}. There are some differences: 
\begin{enumerate}[(i)]
\item We replace the key technical result \cite[4.19(a)]{HV} with 
Theorem \ref{thm:Smale-Hirsch} and 
Theorem \ref{thm:Whitney};
 \item In the even dimensional case extra care must be taken to obtain Poincar\'e embedded handles in the middle
 dimension. This makes use of \cite[thm.~E]{Klein_compression2}, which we restate as Theorem \ref{thm:extension} below.\footnote{Theorem \ref{thm:extension}  is similar to the statement of \cite[4.19(b)]{HV} but the proofs are very different.}
 \end{enumerate}
 
 \begin{proof}[Proof of Theorem \ref{bigthm:handlebody},
$d=2k+1$ odd]
Here $k\ge4$. It suffices to work with one connected component
of $M$ at a time. We construct Poincar\'e embedded stages
$M_j\to M$ such that $M_j$ has homotopy spine dimension at
most $j$ and $M_j\to M$ is $j$-connected.

The initial stages are obtained by realizing a finite
$2$-skeleton of $M$: use $0$-handles for vertices,
$1$-handles for edges, and $2$-handles for the attaching
relations. 
The  immersion and compression constructions appearing below
realize these attachments and give a codimension zero Poincar\'e embedding
$M_2\to M$  whose underlying map is $2$-connected. Furthermore, the pair $(M_2,\partial M_2)$ is $2$-connected.
From this point on, the relative homotopy groups used
in the induction are modules over $\Lambda = \Bbb Z[\pi]$, with $\pi = \pi_1(M)$.

Suppose $3\le j\le k$ and $M_{j-1}$ has been constructed.
Let $C_{j-1}$ be its Poincar\'e complement, so that
\[
\xymatrix{
\partial M_{j-1} \ar[r]\ar[d] & C_{j-1}\ar[d]\\
M_{j-1}\ar[r] & M
}
\]
is a homotopy pushout. The vertical maps are
$(d-j)$-connected. Excision and the relative Hurewicz theorem,
together with the induction hypothesis, imply that
$\partial M_{j-1}\to C_{j-1}$ is $(j-1)$-connected.
Choose relative disk representatives of a finite generating
set of $\pi_j(C_{j-1},\partial M_{j-1})$, and put
\[
(P_j,\partial_0P_j)=\coprod_{i=1}^r
(D^j\times D^{d-j},S^{j-1}\times D^{d-j}).
\]

The disk representatives determine compatible stable normal
data on their boundary spheres. Theorem
\ref{thm:Smale-Hirsch} supplies a Poincar\'e immersion
$\partial_0P_j\to\partial M_{j-1}$ with these data.
Its spine dimension is $j-1$ and its ambient dimension is
$d-1$, so Theorem \ref{thm:Whitney} compresses it to an
embedding. Using this boundary embedding and the same
normal data, Theorem \ref{thm:Smale-Hirsch} supplies a
relative immersion
\[
(P_j,\partial_0P_j)\longrightarrow
(C_{j-1},\partial M_{j-1}).
\]
Theorem \ref{thm:Whitney} compresses this immersion to a
relative embedding. All its numerical hypotheses hold for
$j\le k$ and $d=2k+1$, as do those of the application to the boundary.
Attach these handles to $M_{j-1}$ to obtain $M_j$.
This yields a Poincar\'e embedding
$M_j\to M$ whose underlying map is
$j$-connected. Moreover, the homotopy spine dimension of $M_j$ is  $\le j$.
This completes the forward construction through $M_k$.

Set $C=C_k$, $A=\partial M_k$. The homotopy pushout diagram of the embedding and the
forward connectivity give isomorphisms on fundamental
groups for $C\to M$ and $A\to C$. Every local system
$\cal B$ on $C$ therefore extends to $M$. Using the same
notation for this extension, duality and excision give
\[
H^q(C;\cal B)
\cong H_{d-q}(C,A;\cal B\otimes L_C)
\cong H_{d-q}(M,M_k;\cal B\otimes L_M)=0
\qquad(q>k).
\]

Apply the forward construction inside $C$ only through
$W_{k-1}$, with homotopy spine dimension at most $k-1$
and $W_{k-1}\to C$ $(k-1)$-connected. Let
\[
D_*=C_*(\tilde C,\tilde W_{k-1})
\]
be a finite free relative cellular model, where tildes
denote universal covers. Then $H_i(D_*)=0$ for $i<k$.
The preceding vanishing and the cohomology sequence give
$H^q(C,W_{k-1};\cal B)=0$ for $q>k$.
By \cite[lem.~10.26(ii)]{Ranicki_ags}, $D_*$ is chain
homotopy equivalent to a finitely generated stably free
module
\[
E=H_k(D_*)\cong\pi_k(C,W_{k-1})
\]
concentrated in degree $k$.

Adjoin trivial $(k-1)$-handles to $W_{k-1}$, with their
new core spheres mapped trivially into $C$. This replaces
$E$ by $E\oplus\Lambda^s$, so we may assume $E$ is free.
These handles are included in the index-$(k-1)$ stage.
Now apply the relative immersion and compression
construction to disk representatives of a basis of $E$.
Homotopy excision supplies such representatives in the
complement of $W_{k-1}$, with boundary in
$\partial W_{k-1}$. The resulting $k$-handles complete
an interior Poincar\'e embedded filtration
\begin{equation}\label{eqn:filtration}
\emptyset=W_{-1}\subset W_0\subset\cdots\subset W_k
\to C.
\end{equation}
The map
\[
H_k(\tilde W_k,\tilde W_{k-1})
\longrightarrow H_k(\tilde C,\tilde W_{k-1})
\]
is an isomorphism. The relative exact sequence shows that
$W_k\to C$ is a homology equivalence on universal covers.
It is also an isomorphism on fundamental groups, and hence
is a homotopy equivalence.

Let $T$ be its Poincar\'e complement and set
$B=\partial W_k$. Thus
\[
C\simeq T\cup_B W_k,\qquad \partial T=A\amalg B.
\]
The map $B\to W_k$ is $k$-connected, and hence so is its
homotopy cobase change $T\to C$. Since $W_k\to C$ is an
equivalence, $B\to T$ induces an isomorphism on fundamental
groups. The same holds for $A\to T$, since $A\to C$
already induces such an isomorphism. Excision makes
$(T,B)$ acyclic with all local coefficients; duality does
the same for $(T,A)$. Thus both boundary inclusions into
$T$ are homotopy equivalences.

Choose a homotopy inverse $r\colon T\to A$ to the
cofibration $A\to T$, with $r|_A=\id_A$. Then
\[
R=\id_{M_k}\cup r\colon
(M_k\cup_A T,B)\longrightarrow(M_k,A)
\]
is a homotopy equivalence of pairs. Turning
\eqref{eqn:filtration} upside down gives
\[
T=V_{-1}\subset V_0\subset\cdots\subset V_k=C,
\]
where $V_j$ is the complement of $W_{k-j-1}$ in $C$.
The transition from $V_{j-1}$ to $V_j$ attaches handles
of index $d-k+j=k+j+1$.

After adjoining $M_k$ along $A$, transport these
attachments successively across $R$, beginning with
the boundary equivalence $r|_B\colon B\to A$.
Denote the resulting stages by $M_{k+j+1}$,
$0\le j\le k$. Homotopy invariance of the attaching
pushouts gives
\[
M_{k+j+1}\simeq M_k\cup_A V_j,
\qquad M_d\simeq M_k\cup_A C\simeq M.
\]
Consequently,
\[
\emptyset=M_{-1}\subset M_0\subset\cdots\subset M_d
\simeq M
\]
is the required handle structure.
\end{proof}

\begin{proof}[Proof of Theorem \ref{bigthm:handlebody},
$d=2k$ even]
Here $k\ge4$. The preceding forward construction applies
through $M_{k-1}$, giving a $(k-1)$-connected map
$M_{k-1}\to M$. At the middle stage, use Theorem
\ref{thm:extension} instead of Theorem \ref{thm:Whitney}
to attach $k$-handles killing a finite generating set of
$\pi_k(M,M_{k-1})$. Homotopy excision gives the required
relative disk representatives in the complement.
The attaching boundary embeddings are chosen as part of
each application of Theorem \ref{thm:extension}, not
fixed in advance.

At each step, we transport the attaching data through
the equivalence of pairs supplied by Theorem
\ref{thm:extension}, and we glue the previously
constructed piece using the induced boundary
equivalence. Homotopy invariance of the homotopy
pushouts preserves the ambient homotopy type and
the preceding filtration.

The initial complement pair is $(k-1)$-connected.
After each attachment, the map of the enlarged handlebody
to $M$ remains $(k-1)$-connected and its homotopy spine
dimension is at most $k$. Excision and relative Hurewicz
therefore show that the new complement pair is again
$(k-1)$-connected. In particular, every application of
Theorem \ref{thm:extension} has a $2$-connected ambient
pair. The resulting Poincar\'e embedding $M_k\to M$
is $k$-connected.

Set $C=C_k$ and $A=\partial M_k$. As above, all local
systems on $C$ extend to $M$, and now duality gives
\[
H^q(C;\cal B)
\cong H_{2k-q}(M,M_k;\cal B\otimes L_M)=0
\qquad(q\ge k).
\]
Thus the reverse construction can stop at index $k-1$.
Construct $W_{k-2}\to C$ first. The finite-chain argument
in the odd-dimensional proof, with last index $k-1$,
shows that
\[
\pi_{k-1}(C,W_{k-2})
\cong H_{k-1}(\tilde C,\tilde W_{k-2})
\]
is finitely generated and stably free, and that the
relative chain complex is concentrated in this degree.
Adjoin trivial $(k-2)$-handles to make this module free,
then attach $(k-1)$-handles representing a basis.
These reverse attachments are below the middle dimension,
so Theorem \ref{thm:Whitney} applies. The resulting
$W_{k-1}\to C$ is a homotopy equivalence.

With $T$ its complement and $B=\partial W_{k-1}$,
the same excision, duality, and fundamental-group
argument shows that $A\to T$ and $B\to T$ are
homotopy equivalences. Reverse the filtration of
$W_{k-1}$ and transport the attaching embeddings across
\[
(M_k\cup_A T,B)\simeq(M_k,A)
\]
as before. The reversed handles have indices
$k+1,\ldots,2k$, so they extend the forward filtration
through $M_k$ to the required handle structure on $M$.
\end{proof}

\subsection{The middle dimension} We turn to a  result that was used in the proof of Theorem \ref{bigthm:handlebody} when $d=2k$ is even. 
Suppose that $(M,\partial M)$ is a $2$-connected Poincar\'e pair of dimension $d= 2k$, where $M$ is connected with fundamental group $\pi$ and $\partial M$ is based and non-empty.
Consider the $k$-handle
\[
(H,\partial_0 H,\partial_1 H) := (D^k\times D^{k},S^{k-1} \times D^{k},D^k \times S^{k-1})\, .
\]
Let 
\[
[\alpha]\in \pi_k(M,\partial M)
\]
be a homotopy class.

\begin{thm}[{\cite[thm.~E]{Klein_compression2}}] \label{thm:extension}  Assume in addition that $k \ge 4$. Then there is a Poincar\'e triad 
\[
(N;\partial_0 N,\partial_1 N)
\]
with $\partial_0 H = \partial_0 N$ and
 a weak homotopy equivalence $(N,\partial N) @> {}_\sim >> (M,\partial M)$.

Furthermore, with respect to the isomorphism $\pi_k(N,\partial N) \cong \pi_k(M,\partial M)$, 
the homotopy class $[\alpha]$ is represented by an embedding
\[
H \to N
\]
relative to $\partial_0 H$.
\end{thm}

\begin{rem} Theorem \ref{thm:extension} does not assert that every Poincar\'e immersion representing the homotopy class  $[\alpha]$ compresses to a Poincar\'e embedding.
The proof of \cite[thm.~E]{Klein_compression2} is homotopy theoretic.\end{rem}

\appendix

\section{Top cell splitting of a Poincar\'e normal map} \label{sec:top-cell}
Let $f\: M^n \to X$ be a Poincar\'e normal map, where $X$ and $M$ are connected and $n \ge 3$.  
Set $\pi = \pi_1(X)$ and let $w$ be the orientation character of $X$. We may assume that $(M;P,D^n)$ and $(X;Q,D^n)$
are codimension one splittings defining top cell decompositions. 

\begin{prop} \label{prop:normal-map-top-cell} There is a homotopy of $f$ to a degree-one map of triads
\[
(g;g_0,g_1) \: (M;P,D^n) \to (X;Q,D^n)\, .
\] 
\end{prop}

\begin{rem} When $f$ is a normal map with smooth source, the result may be proved using transversality by means of the Hopf cancellation
trick (cf.~\cite{489175}). 
\end{rem}

\begin{proof}[Proof of Proposition  \ref{prop:normal-map-top-cell}] Let $D^n_M$ and $D^n_X$ denote the top cells of $M$ and $X$ respectively; equip them with the induced orientations.
Let $d\: D^n_M \to D^n_X$ be an orientation-preserving identification. As $D^n_M$ is contractible,
the restriction $f_{|D^n_M}$  is homotopic to the composition
\[
D^n_M @> d >> D^n_X @> \subset >> X\, .
\]
By the homotopy extension property we may arrange it so that  $f_{|D^n_M}$ is this composite.
We transport the stable normal data and the induced
identification of orientation local systems along
this homotopy.

Then $f(P,\partial P) \subset (X,Q)$. Set
\[
h := (f_{|P},f_{|\partial P}) \: (P,\partial P) \to (X,Q)\, .
\]
We shall find a homotopy of
$h$ relative to $\partial P$ to a map having image in $Q$; the resulting map factorizes as
\[
(P,\partial P) \to (Q,\partial Q) @> \subset >> (X,\partial Q)\, .
\]
By Poincar\'e duality and the fact that $Q\to X$ is $(n-1)$-connected, the
only possible non-trivial obstruction $o(h)$ to finding such a factorization lies in the twisted cohomology group
\[
H^n(P,\partial P; h^\ast \cal A)   \, ,
\]
where $\cal A$ is the local system on $X$ associated with the $\Bbb Z[\pi]$-module $\pi_n(X,Q) \cong \Bbb Z[\pi]$
(the isomorphism uses the fact that $X$ is obtained from $Q$ by attaching an $n$-cell). 
It suffices to show that $o(h)$ is trivial.

Note that the map
of locally constant sheaves
\[
w_\ast \: \cal A \to L_X\, ,
\]
induced by $w\: \pi\to \{\pm 1\}$, is a surjection.  

Let
\[
u\in H^n(X,Q; \cal A)
\]
denote the universal primary obstruction, relative to $Q$, to deforming the identity map of $X$ into $Q$. Then
\[
o(h) = h^\ast(u)\, .
\]

After applying the identification $L_P \cong h^\ast L_X$, 
I claim that the induced homomorphism
\[
\epsilon_w  \: H^n(P,\partial P; h^\ast \cal A) @>>> H^n(P,\partial P; L_P)
\]
is an isomorphism. To see this, we use Poincar\'e duality to identify
\begin{align*}
H^n(P,\partial P; h^\ast \cal A)  &\cong H_0(P; h^\ast \cal A \otimes L_P)\, ,  \\
H^n(P,\partial P; L_P)  &\cong H_0(P; L_P \otimes L_P) = \Bbb Z\, .
\end{align*}
Since $f$ is of degree one, it induces a surjection on fundamental groups. As $n \ge 3$, the inclusion $P \to M$ 
is an isomorphism on fundamental groups. Consequently $h_\ast\: \pi_1(P) \to \pi$ is surjective.

The group $H_0(P; h^\ast \cal A \otimes L_P)$ is therefore identified with the group of coinvariants
 $(\Bbb Z[\pi] \otimes \Bbb Z^w)_{\pi}$, where $\pi$ acts diagonally, and $\Bbb Z^w$ is $\Bbb Z$ with
 the action of $\pi$ defined by $w$.

With respect to these identifications, we infer that $\epsilon_w$ is identified with 
the evident homomorphism 
\[
(\Bbb Z[\pi] \otimes \Bbb Z^w)_{\pi} \to (\Bbb Z^w \otimes \Bbb Z^w)_{\pi} 
\]
where on the target $\pi$ acts trivially. It is then readily verified that this map is an isomorphism and the claim is established.

Let $[M] \in H_n(M; L_M)$ be the fundamental class. 
By the relative Mayer-Vietoris sequence we obtain a preferred identification
\begin{equation}  \label{eqn:MV}
H_n(M,\partial P) \cong H_n(P,\partial P) \oplus  H_n(D^n,S^{n-1}) \, ,
\end{equation}
where we have suppressed the local coefficients.

Let
\begin{align*}
r_M \: H_n(M; L_M) &\to H_n(M,\partial P; L_M)\, , \\
r_X \: H_n(X; L_X) & \to H_n(X,Q; L_X) 
\end{align*}
be the evident homomorphisms. We also write
\[
\delta_M \in H_n(D^n_M,\partial D^n_M; \chi_M^*L_M) \quad \text{ and } \quad 
\delta_X \in H_n(D^n_X,\partial D^n_X; \chi_X^*L_X) 
\]
for the compatible fundamental classes on the top cells; here $\chi_M$ and $\chi_X$ denote the characteristic maps of the top cells.

Then the  fundamental classes for $(P,\partial P)$ and $(D^n,S^{n-1})$
glue to the fundamental class $[M]$: 
\[
 r_M([M]) = [P] + \delta_M \, ,
\] 
where the above equation is to be interpreted using \eqref{eqn:MV}.
Note that the
fundamental classes have opposite orientations along $\partial P = S^{n-1}$.
Since $f$ has degree one, 
\begin{align*}
r_X([X]) &= r_X f_\ast([M])\, ,\\
 &=  h_\ast([P]) +  (\chi_X)_\ast d_\ast \delta_M \, , \\
&= h_\ast([P]) +  (\chi_X)_\ast \delta_X \, , \\
&= h_\ast([P]) + r_X([X])\, ,
\end{align*}
where every term lies in $H_n(X,Q;L_X)$.
Consequently, $h_\ast([P]) = 0$.

Set
\[
v := (w_\ast)u \in H^n(X,Q;L_X).
\]
Then by naturality,
\[
\epsilon_w(o(h)) = h^\ast v\, .
\]
Consequently,
\[
\big\langle \epsilon_w(o(h)),[P]\big\rangle
 = \big\langle h^\ast v,[P]\big\rangle
 = \big\langle v,h_\ast[P]\big\rangle
 = 0.
\]
We infer that $\epsilon_w(o(h))=0$, since evaluation at $[P]$ is an isomorphism. Since $\epsilon_w$ is an
isomorphism, $o(h)$ is trivial. Consequently the desired homotopy factorization of $h$ exists. Choose such a
factorization $(g_0,d_{|\partial P}) \: (P,\partial P) \to (Q,\partial Q)$.  Set $g_1 = d$ and $g = g_0 \cup g_1$. Then one has a 
degree-one map of triads \[
(g;g_0,g_1)\: (M;P,D^n_M) \to (X;Q,D^n_X)\, . \qedhere
\]
\end{proof}

\section{Existence of Heegaard splittings} \label{sec:Heegaard-appendix}

\begin{proof}[Proof of Proposition \ref{prop:Heegaard}]  
We will explain the construction in the case $r=1$.  The case $r>1$ is completely analogous and we leave the details to the reader. In this case $(P,\partial_0 P,\partial_1 P)$ is the handle 
\[
(H,\partial_0 H,\partial_1 H) = (D^k \times D^{k+1},S^{k-1} \times D^{k+1}, D^k \times S^k)\, .
\]
Set $x = x_1$.
By  Proposition \ref{prop:normal-map-top-cell},  we may assume $g$ splits as a degree one map of triads 
\[
(g,g_0,g_1)\: (M;M_0,D^{2k+1}) \to (X;X_0,D^{2k+1})\, .
\]
Using the top-cell decompositions, choose a representative
of $x$ consisting of a map
\[
(D^k,S^{k-1})\longrightarrow(M_0,D^{2k})
\]
and a specified null-homotopy of its composite as a map
of pairs into $(X_0,\partial X_0)$. We use this same
null-homotopy for the normal data and the deformation below.

By Theorem \ref{thm:Smale-Hirsch}, the chosen representative
of $x$, with its normal data, underlies a Poincar\'e immersion 
$H \to M_0$ relative to the standard  embedding.
By Theorem \ref{thm:Whitney}  the Poincar\'e immersion is concordant to a  Poincar\'e embedding $H  \to M_0$
relative to the standard embedding. Let $W$ denote the complement of the latter Poincar\'e embedding. Then
\begin{equation} \label{eqn:defnW}
H \cup_{\partial_1 H} W \simeq M_0\, .
\end{equation}
If we set  $U= D^{2k+1} \cup_{\partial_0 H} H$, then 
\begin{itemize}
\item  $(U,\partial U) \cong (S^k \times D^{k+1},S^k\times S^k)$, and
\item $U \cup_{\partial U} W \simeq M$, i.e., $U$ Poincar\'e embeds in $M$ with complement $W$.
\end{itemize}
By the homotopy extension property, the chosen
null-homotopy on $H$ extends to a homotopy of pairs
\[
h_t\:(M_0,\partial M_0)\longrightarrow
(X_0,\partial X_0)
\]
with $h_0=g_0$ and $h_1(H)=\{\ast\}\subset\partial X_0$.
Here we first extend the boundary homotopy over
$\partial M_0$, and then extend over $M_0$.

Extend $h_t|_{\partial M_0}$ across the source top-cell
disk, through maps to the target top-cell disk and
starting at $g_1$. By the homotopy extension property,
this gives a homotopy
\[
q_t\:D^{2k+1}\longrightarrow D^{2k+1}
\]
whose boundary restriction agrees with that of $h_t$.
Gluing gives a homotopy
\[
G_t=h_t\cup q_t\:M\longrightarrow X,
\qquad G_0=g.
\]
Set
\[
f=G_1,\qquad f_0=f|_W,\qquad f_1=f|_U.
\]
Since $f(H)=\{\ast\}$, the map $f$ carries $U$ into
the target top-cell disk and $W$ into $X_0$, with
their common boundary carried into $\partial X_0$.
Transporting the normal data along $G_t$, we obtain
a degree-one map of codimension one splittings
\[
(f;f_0,f_1)\:(M;W,U)\longrightarrow
(X;X_0,D^{2k+1}).
\]
The construction uses the specified representative
of $x$ and its null-homotopy, so this is the required
Heegaard splitting.
  \end{proof}
%

%

%

%


\subsection{On uniqueness}  Although the Heegaard splittings of $f$ are by no means unique, they have cobordant
kernel formations and therefore define the same surgery obstruction:

\begin{prop} \label{prop:Heegaard-unique} The kernel formations associated with any pair of Heegaard splittings of $f$ are cobordant. Consequently,
the surgery obstruction of a Poincar\'e normal map $f\: M^{2k+1} \to X^{2k+1}$ is well defined.
\end{prop}

\begin{proof} This follows from the last part of the proof of \cite[prop.~5.4, final paragraph]{Ranicki}, which 
applies as well to $k$-connected normal maps in the Poincar\'e category.  It asserts that
the cobordism class of the kernel formation associated with a Heegaard splitting
of $f$ coincides with the cobordism class of the quadratic algebraic Poincar\'e complex given by  the algebraic mapping cone of
the (singular) chain map
$C(\tilde f^!)\: C_\ast(\tilde X_+) \to C_\ast(\tilde M_+)$.
The  latter cobordism class
is defined without requiring $f$ to be $k$-connected, is an invariant of normal bordisms, and is defined without resorting to a Heegaard splitting.
\end{proof}


\end{document}